\documentclass[a4paper,10pt]{article}

\usepackage{a4wide}
\usepackage[colorlinks]{hyperref}
\usepackage{amsmath}
\usepackage{amsthm}
\usepackage{amsfonts}
\usepackage{amssymb}
\usepackage{mathtools}
\usepackage{algorithm}
\usepackage{algorithmic}
\usepackage{epsfig}
\usepackage{subfig}
\usepackage[utf8]{inputenc}
\usepackage{graphics}
\usepackage{caption}
\usepackage{xcolor}
\usepackage{booktabs}
\usepackage{tcolorbox}
\usepackage{listings}
\usepackage{color} 
\usepackage{enumitem}

\definecolor{mygreen}{RGB}{28,172,0} 
\definecolor{mylilas}{RGB}{170,55,241}

\usepackage{amsmath}

\DeclareMathOperator*{\esssup}{ess\,sup}

\usepackage{fancyhdr} 
\newcommand{\bx}{{\boldsymbol{x}}}
\newcommand{\cD}{{\mathcal{D}}}
\newcommand{\cK}{{\mathcal{K}}}
\newcommand{\cM}{{\mathcal{M}}}
\newcommand{\E}{{\mathbb{E}}}
\renewcommand{\P}{{\mathbb{P}}}
\newcommand{\R}{{\mathbb{R}}}
\newcommand{\N}{{\mathcal{N}}}
\renewcommand{\d}{\textup{d}}
\newcommand{\uL}{{L^2(\Gamma;L^2(\partial \cD))}}
\newcommand{\uH}{{L^2(\Gamma;H^{1/2}(\partial \cD))}}
\newcommand{\yL}{{L^2(\Gamma;L^2(\cD))}}

\newcommand{\yH}{{L^2(\Gamma;H^1(\cD))}}

\newtheorem{theorem}{Theorem}[section] 
\newtheorem{lemma}[theorem]{Lemma}
\newtheorem{assumption}[theorem]{Assumption}
\newtheorem{corollary}[theorem]{Corollary}
\newtheorem{remark}[theorem]{Remark}

\numberwithin{theorem}{section}

\author{Max Winkler\footnote{Faculty of Mathematics, Chemnitz University of Technology, Germany, \url{max.winkler@mathematik.tu-chemnitz.de}} \and Hamdullah Yücel \footnote{Institute of Applied Mathematics, Middle East Technical University, Ankara, Türkiye, \url{yucelh@metu.edu.tr}}}
\title{A stochastic Galerkin method for optimal Dirichlet boundary control problems with uncertain data}

\begin{document}

\maketitle

\begin{abstract}
The paper deals with a stochastic Galerkin approximation of elliptic Dirichlet boundary control problems with random input data. The expectation of a tracking cost functional with the deterministic constrained control is minimized. Error estimates are derived for the control variable in $L^2(\partial \cD)$-norm and state variable in $\yL$-norm. To solve large linear systems, appropriate preconditioners are proposed for both unconstrained and constrained scenarios. To illustrate the validity and efficiency of the proposed approaches, some numerical experiments are performed.
\end{abstract}

\section{Introduction}\label{sec:intro}

While solving physical and engineering systems, ones  observes discrepancies between the results of simulations and real systems, caused by different kinds of errors which are model errors, numerical errors, and data errors. Some of the above errors may be reduced, e.g., by building more accurate models, by improving numerical resolution methods, and by additional measurements, etc.  Such kinds of  errors or uncertainties that can be reduced are usually called epistemic or systematic uncertainties. There are also other sources of randomness that are intrinsic to the system itself (e.g., uncertainty in the quantum systems) and hence cannot be reduced. They are the so-called aleatoric or statistical uncertainties. Due to the aforementioned reasons, the uncertainty should be taken into account in the mathematical formulation of real systems if one aims at obtaining more reliable numerical predictions from mathematical models. For instance, an oil company may choose to inject water or other solvents into a reservoir to increase pressure and produce more oil but of course the subsurface rock properties are unknown. If the optimization problem is to determine the well locations and injection rates that maximize the net present value of the reservoir, the optimal rates  should be resilient to the inherent uncertainties of the subsurface. While constructing the mathematical formulation, uncertainties can be characterized by parameters associated with, for example, coefficients, boundary data, initial conditions, or source terms that are not known, and  consequently the solution of underlying formulation becomes a random variable or more general a random field. 

To characterize and control the variability of the solution, optimal control problems containing uncertainty governed by partial differential equations (PDEs) have been studied in various formulations in the last decade, such as mean--based control \cite{ABorzi_2010a,ABorzi_VSchulz_CSchillings_GvonWinckel_2010a}, pathwise control \cite{FNegri_AManzoni_GRozza_2015,AAAli_EUllman_MHinze_2017}, average control \cite{MLazar_EZuazua_2014,EZuazua_2014}, robust deterministic control \cite{SGarreis_MUlbrich_2017,MDGunzburger_HCLee_JLee_2011a,LSHou_JLee_HManouzi_2011,DPKouri_MHeinkenschloss_DRidzal_BGWaanders_2013,HCLee_JLee_2013,ERosseel_GNWells_2012,PCiloglu_HYucel_2023}, and   robust stochastic control \cite{PBenner_AOnwunto_MStoll_2016,PChen_AQuarteroni_GRozza_2013,AKunoth_CSchwab_2016,HTiesler_RMKirby_DXiu_TPreusser_2012a,JMFrutos_MKessler_AMunch_FPeriago_2016}. However, the vast majority of the literature about PDE-constrained optimization problems containing uncertainty  is for distributed optimal control problems; see, e.g., \cite{ABorzi_2010a,FNegri_AManzoni_GRozza_2015,AAAli_EUllman_MHinze_2017,EZuazua_2014,SGarreis_MUlbrich_2017,
LSHou_JLee_HManouzi_2011,DPKouri_MHeinkenschloss_DRidzal_BGWaanders_2013,HCLee_JLee_2013,ERosseel_GNWells_2012,PBenner_AOnwunto_MStoll_2016,AKunoth_CSchwab_2016,HTiesler_RMKirby_DXiu_TPreusser_2012a} and there exists limited work on the numerical solution of boundary optimal control problems; see, \cite{MDGunzburger_HCLee_JLee_2011a} for Neumann boundary control, \cite{PChen_AQuarteroni_GRozza_2013,JMFrutos_MKessler_AMunch_FPeriago_2016}  for Robin boundary control, and \cite{WZhao_MGunzburger_2023} for Dirichlet boundary control. Therefore, we are here interested in Dirichlet boundary control problems with random inputs of the form
\begin{eqnarray}\label{eqn:objec}
	\min \limits_{u \in U^{\textup{ad}}} \; \mathcal{J}(y,u) := \frac{1}{2} \int_{\Omega} \int_{\cD} (y(\bx, \omega)-y^d(\bx))^2 \, \d \mathbb{P}(\omega) \, \d\bx +  \frac{\alpha}{2} \int_{\partial \cD}  u^2(\bx) \,  \d s
\end{eqnarray}
subject to 
\begin{subequations}\label{eqn:constPDE}
  \begin{align}
    -\nabla \cdot (a(\bx,\omega) \nabla y(\bx, \omega)) &= f(\bx,\omega)               &&       \text{for}\ (\bx,\omega) \in \cD         \times \Omega,\\
     y(\bx, \omega) &= u(\bx) + b(\bx,\omega)  &&       \text{for}\ (\bx,\omega) \in \partial\cD \times \Omega,
  \end{align}
\end{subequations}
where the first term in \eqref{eqn:objec} is a measure of the distance between the state variable $y$ and the desired state $y^d$ in terms of expectation of $y-y^d$. Without loss of generality, we assume that the state $y$ is a random field, whereas the desired state $y^d$ is modelled deterministically. The second term  corresponds to a deterministic Dirichlet boundary control with a positive regularization parameter $\alpha >0 $.  Further, the closed convex admissible control set is defined by
\begin{equation}\label{defn:admis_space}
	U^{\textup{ad}} :=  \{ u \in L^2(\partial \cD) : \; u_a \leq  u(\boldsymbol{x}) \leq u_b, \;\; \forall \boldsymbol{x} \in \partial \mathcal{D} \},
\end{equation}
where the constant control bounds $u_a, u_b \in \mathbb{R}$ fulfill $u_a \leq u_b$. Last, we note that the gradient notation, $\nabla$, in \eqref{eqn:constPDE}, always refers to differentiation with respect to $\bx \in \cD$, unless otherwise stated.

Even for the deterministic setting,  Dirichlet boundary control problems involve some specific difficulties resulting from the fact that they are not of variational type when the control belongs to $L^2(\partial \cD)$. We refer to \cite{TApel_MMateos_JPfefferer_ARosch_2015,TApel_MMateos_JPfefferer_ARosch_2018,MBerggren_2004a,ECasas_JPRaymond_2006b,KDeckelnick_AGunther_MHinze_2009,MDGunzburger_SManservisi_2000a,SMay_RRannacher_BVexler_2013,MWinkler_2020} and references therein for a comprehensive understanding. When it comes to Dirichlet boundary control problems with random input data, the construction and analysis become more complex compared to its analogues in the deterministic setting. Solving strategies of stochastic optimal control problems in the literature are based on either sampling  or a discretization of the stochastic space, in particular, Monte Carlo (MC) \cite{AAAli_EUllman_MHinze_2017,PAGuth_VKaarnioja_FYKuo_CSchillings_IHSloan_2021a,AVBarel_SVandewalle_2019}, stochastic collocation (SC) \cite{ABorzi_GvonWinckel_2009a,HTiesler_RMKirby_DXiu_TPreusser_2012a,ERosseel_GNWells_2012}, or stochastic Galerkin method (SG) \cite{IBabuska_RTempone_GEZouraris_2004a,HCLee_JLee_2013,ERosseel_GNWells_2012}. To discretize the stochastic space, we use the stochastic Galerkin method, a variant of generalized polynomial chaos approximation \cite{DXiu_GEKarniadakis_2002}. In contrast to MC and SC, the SG  method is a nonsampling method and is based on the orthogonality of the  residuals into the polynomial space. Since the spatial and (stochastic) parametric spaces are separable,  the established numerical techniques in the deterministic setting can be adopted easily. Because of that, we utilize the stochastic Galerkin method to address Dirichlet boundary control problems with uncertain input data \eqref{eqn:objec}-\eqref{eqn:constPDE}.

The novelty in this study is that a stochastic Galerkin approximation scheme for a Dirichlet boundary control problem governed by an elliptic PDE  with random data inputs is investigated. The control objective is to minimize the expectation of a tracking cost functional with the deterministic constrained control. By using the very weak formulation \cite{ECasas_JPRaymond_2006b,KDeckelnick_AGunther_MHinze_2009} for the governing state equation \eqref{eqn:constPDE} within the physical domain and $L^2(\partial \cD)$ as control space, we  discuss the existence of a unique solution to the optimization problem \eqref{eqn:objec}-\eqref{eqn:constPDE}.  By leveraging the stochastic Galerkin method under the  finite dimensional noise assumption \cite{NWiener_1938a}, we convert our stochastic problem into a high-dimensional deterministic one. Then, we apply standard piecewise linear and continuous finite elements  for the discretization of the state, the adjoint state, and the control in the spatial domain. To deduce optimal error estimates, we first derive a priori error estimates for derivatives of the stochastic unknowns, that is, state $y(\bx, \xi)$ and adjoint state $p(\bx, \xi)$  with respect to the stochastic variable $\xi\in \Gamma$ for each $n=1,\ldots,N$ and $q_n\in \mathbb N_0$
\[
    \frac{\|\partial_{\xi_n}^{q_n+1} v(\cdot,\xi)\|_{H^1(\cD)}}{(q_n + 1)!}
     \leq C\,r_n^{q_n+1},
\]
cf. Lemma~\ref{lem:regularity_derivatives_xi}, where $r_n = \|a_n\|_{L^\infty( \cD)}/a_{\min}$ with $a_n$ the deterministic modes of a Karhunen--Lo\`{e}ve expansion of the diffusion coefficient $a$. This regularity estimate allows us to derive even error estimates for the stochastic Galerkin method in the $L^2(\Omega;L^2(\cD))$-norm, complementing existing results from \cite{IBabuska_RTempone_GEZouraris_2004a}. As a consequence, by integrating the above estimate with the projection estimates and a quasi-interpolant construction of the optimal control \cite{TApel_MMateos_JPfefferer_ARosch_2018,ECasas_JPRaymond_2006b} we derive our main result presented in Theorem~\ref{thm:final_estimate}
\[
  \|u-u_h\|_{L^2(\partial \cD)} + \|y-y_h^\gamma\|_{L^2(\Gamma;L^2(\cD))}\le C\left(h^{1/2}
  +  h^{-1/2}\sum \limits_{n=1}^N \left(\frac{\gamma_n\,r_n}{2}\right)^{q_n+1}
  \right),
\]
where $u_h$ and $y_h^\gamma$ are discrete  optimal control and state, respectively. Here,  $h$  and $\gamma_n$  refer to the mesh sizes in the spatial and stochastic (parametric) domains, whereas $q_n$   indicates the polynomial degree for each direction  $\xi_n$ within a multi-index $q = (q_1,\ldots,q_N)$. This result is  a worst-case estimate only which is sharp in case that the solution is as regular as one would expect on general convex polygonal domains $\cD$. In addition to these theoretical findings, we explicitly construct the saddle-point system and propose suitable preconditioners based on the  Schur complement to solve the corresponding linear systems. Numerical experiments indicate that we achieve exponential convergence concerning the polynomial degree  $Q$  of the discretization in the stochastic space, provided that the spatial discretization error does not dominate the total error. Further, we improve the convergence rate of the control  error to  $3/2$  with respect to $h$ by designing meshes that exhibit the superconvergence property \cite[Section 4]{TApel_MMateos_JPfefferer_ARosch_2018}.

The rest of this paper is organized as follows. Section~\ref{sec:model} presents the derivation of the first-order optimality conditions for the optimization problem \eqref{eqn:objec}-\eqref{eqn:constPDE}, given the relevant assumptions and notation. In section~\ref{sec:finite_dimension}, we convert our stochastic problem into a high-dimensional deterministic problem using the stochastic Galerkin method, and establish a finite element approximation for the corresponding optimality system. Section~\ref{sec:error_estimates} is devoted to the proof of the error estimates. This includes our primary result, Theorem~\ref{thm:final_estimate}, which explains the observed orders of convergence in terms of the mesh sizes and  polynomial orders. We present the matrix representation of the discretized optimality system  and discuss appropriate preconditioners for solving the related linear systems
in section~\ref{sec:linear_system}. In section~\ref{sec:numeric}  we provide some numerical experiments to support   our theoretical findings  and  to test the performance of the proposed preconditioners.
Last, some conclusion and concluding remarks are presented in section~\ref{sec:conclusion}.


\section{Stochastic Dirichlet boundary control problem}\label{sec:model}

Throughout this paper, we follow the standard notation (see, e.g., \cite{RAAdams_1975})  for Sobolev spaces $H^k(\mathcal{D})$ equipped with the norm $\| \cdot \|_{H^k(\mathcal{D})}$ and the seminorm  $| \cdot |_{H^k(\mathcal{D})}$ on a convex polygonal physical domain $\mathcal{D} \subset \mathbb{R}^d$ $(d=1,2,3)$ with Lipschitz boundary $\partial \mathcal{D}$. Inner products  on $H^k(\mathcal{D})$ and $H^k(\partial \mathcal{D})$ are denoted by $( \cdot,\cdot )_{H^k(\mathcal{D})}$ and $( \cdot,\cdot )_{H^k(\partial \mathcal{D})}$, respectively, and $C>0$ is a generic constant independent of any discretization parameters. A complete probability space is represented by $(\Omega, \mathfrak{F}, \mathbb{P})$, where $\Omega$  is a set of outcomes $\omega \in \Omega$, $\mathfrak{F}$ is a $\sigma$-algebra of events, and $\mathbb{P}: \mathfrak{F} \rightarrow [0,1]$ is the probability measure. For a random variable $v: \Omega \rightarrow H^k(\mathcal{D})$, the Bochner space, $L^p(\Omega;H^k(\mathcal{D}))$, is then defined by
\begin{equation*}
L^p(\Omega;H^k(\mathcal{D})): = \{v: \Omega \rightarrow H^k(\mathcal{D}) : v \; \; \text{strongly measurable}, \;  \|v \|_{L^p(\Omega;H^k(\mathcal{D}))} < \infty  \},
\end{equation*}
where
\[
 \|v \|_{L^p(\Omega;H^k(\mathcal{D}))}  = \left\{
                                \begin{array}{ll}
                                     \left( \int_{\Omega} \|v(\omega)\|^p_{H^k(\mathcal{D})} \, d\mathbb{P}(\omega) \right)^{1/p}, & \hbox{for  }  1 \leq p < \infty,\\
                                      \esssup \limits_{\omega \in \Omega} \|v(\omega)\|_{H^k(\mathcal{D})}, & \hbox{for  } p = \infty.
                                \end{array}
                             \right.
\] 
Moreover, we have the following isomorphism relation \cite{IBabuska_RTempone_GEZouraris_2004a}:
\[
    H^k(\mathcal{D}) \otimes L^p(\Omega) \simeq L^p(\Omega;H^k(\mathcal{D})) \simeq  H^k(\mathcal{D};L^p(\Omega)).
\]
Further, for a real--valued random field  $z: \cD \times \Omega \rightarrow \mathbb{R}$ in the probability space $(\Omega, \mathfrak{F}, \mathbb{P})$,  expected value  (also called as mean) and covariance  are given, respectively, by 
\begin{subequations}\label{eqn:cov}
	\begin{eqnarray}
		\mathbb{E}[z](\bx)&:=& \int_{\Omega} z(\bx,\omega) \, d\mathbb{P}(\omega) \qquad \quad \; \bx \in \mathcal{D}, \\
		\mathcal{C}_{z}(\bx,\widetilde{\bx}) &:=&  \int_{\Omega} (z(\bx,\omega) - \mathbb{E}[ z ](\bx)) (z(\widetilde{\bx},\omega) - \mathbb{E}[ z ](\bx) ) \, d\mathbb{P}(\omega) \quad \;\;  \bx,\widetilde{\bx} \in \mathcal{D}. \label{eqn:covb}
	\end{eqnarray}
\end{subequations} 
Setting $\bx=\widetilde{\bx}$ yields the variance $\mathcal{V}_{z}(\bx) = \mathcal{C}_{z}(\bx,\bx)$ and standard deviation $\kappa_{z}= \sqrt{\mathcal{V}_{z}}$.

To ensure the regularity of the solution $y$ in the state equation \eqref{eqn:constPDE}, we make the following assumptions on the given random inputs:
\begin{assumption}\label{asm:a_coef}
   There exists positive constants $0 < a_{\min} < a_{\max} < \infty$  such that  the diffusion coefficient satisfies
     \[
         a_{\min} \leq  a(\boldsymbol{x},\omega) \leq a_{\max}
     \]
   for almost every  $(\boldsymbol{x},\omega) \in  \mathcal{D} \times \Omega$ and the coefficient  $a(\cdot,\omega)$ belongs to  $C^1(\cD)$.
\end{assumption}
\begin{assumption}\label{asm:fb_coef}
   The stochastic functions  $f(\bx,\omega)$ and $b(\bx,\omega)$ have continuous and bounded covariance functions and belong to 
     \[
         b \in  L^2(\Omega; H^{1/2}(\partial\cD)) \quad \text{and} \quad f \in  L^2(\Omega; L^{2}( \cD)).
     \]
    Moreover, the desired state is defined in a deterministic way and fulfills $y^d \in L^2(\cD)$. 
\end{assumption}
\noindent Then, by the bilinear form 
\begin{equation*}
  a[y,v] = \int_\Omega \int_\cD a(\bx,\omega)\, \nabla y(\bx,\omega)\cdot \nabla v(\bx,\omega)\,\d \bx\,\d\P(\omega)
\end{equation*}
and scalar products
\begin{align*}
  [y,v] &= \int_\Omega \int_\cD y(\bx,\omega)\, v(\bx,\omega)\,\d \bx\,\d\P(\omega), \\
  [y,v]_{\partial\cD} &= \int_\Omega\int_{\partial\cD} y(\bx,\omega)\, v(\bx,\omega)\,\d s_\bx\,\d\P(\omega),
\end{align*}
the standard weak formulation of the state equation \eqref{eqn:constPDE} for a fixed control $u\in H^{1/2}(\partial\cD)$ can be written as 
\begin{equation}\label{eq:weak_satte}
 y|_{\partial\cD}= b+u,\quad a[y,v] = [f,v] \qquad \forall \ v\in L^2(\Omega; H_0^1(\cD)).
\end{equation}

However, due to the low \emph{a priori} regularity of the control $u\in U:= L^2(\partial \cD)$ in the optimal control problem, we must understand  the state equation \eqref{eqn:constPDE} in a very weak sense first, i.\,e., we seek $y \in L^2(\Omega; L^2(\cD))$ satisfying
\begin{align}\label{eqn:state_weak}
  \int_\Omega \int_\cD - y \, \nabla \cdot (a\, \nabla v) \,\d \bx \, \d \P & + \int_\Omega \int_{\partial\cD} (b+u)\, a \,  \partial_n v\,\d s_\bx
  \d \P = \int_\Omega \int_\cD f\,v \,\d \bx \, \d \P
\end{align}
for all $v\in L^2(\Omega; H_0^1(\cD) \cap H^2(\cD))$ with $\partial_n v := \nabla v \cdot \mathbf{n}$ and $\mathbf{n}$ representing the unit outward normal vector to $\partial \cD$. For each $u\in U^{\textup{ad}}$ there exists a unique solution $y$ which even belongs to $L^2(\Omega; H^{1/2}(\cD))$, provided that $\cD$ is a convex polygonal domain; see, e.g., \cite{ECasas_1985} for a similar argumentation in the deterministic setting. This allows to introduce the control-to-state operator $S\colon U^{\textup{ad}} \to L^2(\Omega; L^2(\cD))$ defined by  $u\mapsto S(u) := y$, and one confirms by implicit differentiation that $S$ is of class $C^\infty$ and $\widetilde y = S'(u) \widetilde u$ is characterized by the very weak formulation
\begin{equation*}
  \int_\Omega\left[-\int_\cD \widetilde y \, \nabla \cdot (a\, \nabla v) \, \d \bx + \int_{\partial\cD} \widetilde u \, a \,  \partial_n v \, \d s_\bx \right] \d\P = 0
\end{equation*}
for all $v\in L^2(\Omega; H_0^1(\cD) \cap H^2(\cD))$.  With the chain rule we find that the objective of the reduced optimization problem $j(u) := \mathcal{J}(S(u),u)\to\min!$ is G\^ateaux differentiable.
Therefore, we introduce the adjoint equation for an adjoint state $p\in L^2(\Omega;H_0^1(\cD)\cap H^2(\cD))$ by
\begin{equation}\label{eq:adjoint}
  -\int_\Omega \int_\cD \nabla \cdot (a\, \nabla p) \, v \,\d \bx\,\d\P = \int_\Omega\int_\cD (y-y^d)\,v\,\d\bx\,\d\P
\end{equation}
for all $v\in L^2(\Omega; L^2(\cD))$ and by choosing 
$v=\widetilde y = S'(u)\widetilde u$ in \eqref{eq:adjoint} we deduce
\begin{align}\label{eq:derivative_reduced_begin}
  j'(u)\widetilde u &= \int_\Omega \left(\int_\cD (y - y^d)\,\widetilde y\,\d \bx
                  + \alpha\,\int_{\partial\cD}u\,\widetilde u\,\d s_\bx \right)\d\P
                  \nonumber\\
                &= \int_\Omega\left( -\int_\cD \widetilde y \, \nabla \cdot (a\, \nabla p) \,\d \bx
                  + \alpha\,\int_{\partial\cD} u\,\widetilde u\,\d s_\bx\right)\d\P
                  \nonumber\\
                &= \int_\Omega
                  \int_{\partial\cD} (\alpha\,u - a \, \partial_n p)\,\widetilde u\,\d s_\bx\,\d\P.
\end{align}
From \eqref{eq:derivative_reduced_begin}, it is observed that the variational inequality $j'(u)(\widetilde u-u) \geq 0$ for all $\widetilde u \in  U^{\textup{ad}}$ links the adjoint state to the control by means of $u = \mathcal P_{U^{\textup{ad}}} \left(\frac{1}{\alpha} \, \mathbb{E}[a\, \partial_n p] \right)$, where $\mathcal P_{U^{\textup{ad}}}$ denotes the pointwise projection onto the admissible set $U^{\textup{ad}}$. By standard bootstrapping arguments we can show that the optimal control and the optimal state must be even more regular:
\begin{align}\label{eq:bootstrapping}
  u \in U^{\textup{ad}}
  &\quad\Rightarrow\quad
  y\in L^2(\Omega;L^2(\cD))
  \quad\Rightarrow\quad
  p\in L^2(\Omega; H_0^1(\cD) \cap H^2(\cD) ) \nonumber\\
  &\quad\Rightarrow\quad 
  u = \mathcal P_{U^{\textup{ad}}} \left(\frac{1}{\alpha} \,\mathbb{E}[a\,\partial_n p] \right) \in L^2(\Omega; H^{1/2}(\partial\cD)) \nonumber\\
  &\quad\Rightarrow\quad
  y \in L^2(\Omega; H^1(\cD)),
\end{align}
which allows us to consider the state equation \eqref{eqn:constPDE} also in the standard weak form. We summarize  our investigations above in the following theorem:
\begin{theorem}
  For each given $y^d\in L^2(\cD)$, $f\in L^2(\Omega;L^2(\cD))$, $b\in L^2(\Omega; H^{1/2}(\partial \cD))$, and $\alpha>0$, the optimization problem \eqref{eqn:objec}--\eqref{eqn:constPDE} possesses a unique solution pair $(y,u)\in L^2(\Omega; H^1(\cD))\times U^{\textup{ad}}$ and there exists an adjoint state $p\in L^2(\Omega; H_0^1(\cD)\cap H^2(\cD))$ satisfying the following optimality system:
  \begin{subequations}\label{eq:opt_sys}
    \begin{alignat}{2}
        y|_{\partial\cD}= b+u,\quad a[y,v] &= [f,v] &\qquad &\forall v\in L^2(\Omega; H_0^1(\cD)),\label{eq:opt_sys_a}\\
        a[q,p] &= [y-y^d,q] &&\forall q\in L^2(\Omega; H_0^1(\cD)),\label{eq:opt_sys_b}\\
        (\alpha\,u - \E[a \,\partial_n p], w-u)_{\partial \cD} &\geq 0 &&\forall w \in U^{\textup{ad}}. \label{eq:opt_sysc}
      \end{alignat}
  \end{subequations}
\end{theorem}
\noindent The variational inequality \eqref{eq:opt_sysc}  is also equivalent to 
\begin{equation}\label{eq:variationalform_cont}
	(\alpha \, u, w-u)_{\partial \cD} + [y-y^d, y(w) - y]  \geq 0 \qquad \forall w \in U^{\textup{ad}},
\end{equation}
cf.\ \eqref{eq:derivative_reduced_begin}, where $y(w)$ is the solution of the very weak problem \eqref{eqn:state_weak} when $u$ is replaced by $w \in U^{\textup{ad}} \subset L^2(\partial \cD)$.


\section{Reduction to finite dimension}\label{sec:finite_dimension} 

To solve the optimization problem containing randomness \eqref{eqn:objec}--\eqref{eqn:constPDE} numerically, we need to reduce it into a finite dimensional setting. Now, we first discuss the representation of random fields by employing a finite dimensional noise assumption \cite{NWiener_1938a} and then provide  the stochastic Galerkin method with  appropriate functional spaces \cite{IBabuska_RTempone_GEZouraris_2004a}.

\subsection{Representation of random fields}

To characterize the given random coefficients and functions with known covariance information, we employ a finite Karhunen--Lo\`{e}ve (KL) expansion \cite{KKarhunen_1947,MLoeve_1946} as follows 
\begin{equation}\label{eqn:kltrun}
	z(\boldsymbol{x},\omega) \approx z_{N}(\boldsymbol{x},\omega) := \overline{z}(\boldsymbol{x}) + \kappa_{z} \,\sum \limits_{k=1}^{N} \sqrt{\lambda_k}\,w_k(\boldsymbol{x})\,\xi_k(\omega).
\end{equation}
In \eqref{eqn:kltrun}, $\overline{z}(\boldsymbol{x})$ and $\kappa_{z}$ are   mean and standard deviation of the random field $z$, respectively. $\xi:=\{\xi_i(\omega)\}_{i=1}^{N}$ are mutually uncorrelated random variables with zero mean and unit variance, i.e., $\mathbb{E}[\xi_i]=0$ and $\mathbb{E}[\xi_i \xi_j]= \delta_{i,j}$. A sequence of eigenpairs $\{\lambda_k,w_k\}_{i=1}^{N}$ obtained by solving the Fredholm integral equation with the kernel $\mathcal{C}_{z}$ consists of non-increasing eigenvalues $\lambda_1 \geq  \ldots \geq  \lambda_k \geq \ldots >0$ and mutually orthonormal  eigenfunctions. The value of $N$, called truncation number, is determined according to the decay of the eigenvalues $\lambda_k$.  To satisfy the positivity requirement for the diffusivity coefficient $a(\boldsymbol{x},\omega)$ in the Assumption~\ref{asm:a_coef},  we further impose a similar condition  as  follows 
\begin{equation}\label{assumption_truncatedD}
  \exists \, a_{\min}, \, a_{\max} >0, \quad \text{s.t.} \quad a_{\min} \leq  a_{N}(\boldsymbol{x},\omega) \leq a_{\max}
\end{equation}
for almost every  $(\boldsymbol{x},\omega) \in  \mathcal{D} \times \Omega$.

To represent the stochastic solutions, i.e., $y$ and $p$,  in the finite setting, we recall the  finite dimensional noise assumption \cite{NWiener_1938a}.

\begin{assumption} \label{asm:finite_noise}
Any general second order random process $X(\omega)$ can be characterized by
finite uncorrelated components $\{\xi_i(\omega)\}_{i=1}^{N}$, $N \in \mathbb{N}$ having zero mean and unity variance. Here, the probability density functions of $\{\xi_i(\omega)\}_{i=1}^{N}$ are denoted by $\rho_i \, : \, \Gamma_i \rightarrow \mathbb{R}^{+}_0$, where $\Gamma_i = \xi_i (\Omega) \subset \mathbb{R}$  are bounded intervals. Then, the corresponding probability space becomes $(\Gamma, \mathcal{B}(\Gamma), \rho(\xi)\d\xi)$, where $\Gamma = \prod \limits_{i=1}^N \Gamma_i \subset \mathbb{R}^N$ represents the support of such probability density,  $\mathcal{B}(\Gamma)$ is a Borel $\sigma$--algebra, $\rho(\xi)=\prod \limits_{i=1}^N \rho_i(\xi_i)$ is the joint probability density function, and $\rho(\xi)\d\xi$ corresponds to the probability measure of $\xi$. 
\end{assumption} 
By Assumption~\ref{asm:finite_noise} and the Doob--Dynkin Lemma \cite{BOksendal_2003}, the stochastic solutions $y$ and $p$ in \eqref{eq:opt_sys} are expressed  by a finite number of random variables, that is,  $y(\boldsymbol{x},\omega)= y(\boldsymbol{x},\xi_1(\omega), \ldots, \xi_N(\omega)) \in L^2(\Gamma; H^1(\cD))$ and $p(\boldsymbol{x},\omega)= p(\boldsymbol{x},\xi_1(\omega), \ldots, \xi_N(\omega)) \in L^2(\Gamma; H_0^1(\cD)\cap H^2(\cD))$. The optimality system \eqref{eq:opt_sys} reduces to a finite parametric space
  \begin{subequations}\label{eq:opt_p}
    \begin{alignat}{2}
        y|_{\partial\cD}= b_{N} + u,\quad a[y,v]_{\rho} &= [f_N,v]_{\rho} &\qquad &\forall v\in L^2(\Gamma; H_0^1(\cD)),\label{eq:opt_pa}\\
        a[q,p]_{\rho} &= [y-y^d,q]_{\rho} &&\forall q\in L^2(\Gamma; H_0^1(\cD)), \label{eq:opt_pb}\\
        (\alpha\,u - \E[a_{N}\,\partial_n p], w-u)_{\partial \cD} &\geq 0 &&\forall w \in U^{\textup{ad}}, \label{eq:opt_pc}
      \end{alignat}
  \end{subequations}
where
\begin{align*}
  & a[y,v]_{\rho} = \int_\Gamma \int_\cD a_N \, \nabla y\cdot \nabla v\,\d \bx\, \rho\, \d\xi, \\
  & [y,v]_{\rho} = \int_\Gamma \int_\cD y\, v\,\d \bx\,\rho\, \d\xi, \qquad  (y,v)_{\partial\cD} =  \int_{\partial\cD} y\, v\,\d s_\bx.
\end{align*}
Analogous to \eqref{eq:variationalform_cont}, the optimality condition in \eqref{eq:opt_pc} can be stated as
\begin{equation}\label{eq:variationalform_parametric}
    ( \alpha \, u, w-u)_{\partial \cD} + [y-y^d, y(w) - y]_{\rho}  \geq 0 \qquad \forall w \in U^{\textup{ad}},
\end{equation}
where $y(w)$ is the solution of the problem \eqref{eqn:state_weak} on the finite parametric domain $\Gamma \subset \mathbb{R}^N$ when $u$ is replaced by $w$ and the random coefficients $a, b$ and the random function $f$ are replaced by the truncated ones  $a_N, b_N$, and $f_N$, respectively.

We conclude this section with some \emph{a priori} estimates for the solutions of \eqref{eq:opt_pa} and \eqref{eq:opt_pb}. In particular derivatives with respect to the stochastic variable are required to deduce optimal error estimates.
\begin{lemma}
  \label{lem:regularity_derivatives_xi}
  For given $g_N= \bar g + \sum_{i=1}^N\xi_i\,g_i\in C^\infty(\Gamma; H^1(\cD))$ and $F\in C^\infty(\Gamma; L^2(\cD))$ let $v\in L^2(\Gamma; H^1(\cD))$ be the solution of
  \begin{equation}\label{eq:general_bvp}
    v|_{\partial \cD} = g_N|_{\partial \cD},\qquad a[v,w]_\rho = [F,w]_\rho\quad \forall w\in L^2(\Gamma; H_0^1(\cD)).
  \end{equation}
  Then, for each $n=1,\ldots,N$ and $q_n\in \mathbb N_0$ there holds 
  \begin{align*}
    \frac{\|\partial_{\xi_n}^{q_n+1} v(\cdot,\xi)\|_{H^1(\cD)}}{(q_n + 1)!}
    &\le C\,\sum_{j=0}^{q_n+1} \frac{1}{j!}\,r_n^{q_n+1-j}\,\|\partial_{\xi_n}^j F(\cdot,\xi)\|_{L^2(\cD)} \\
    &\quad + r_n^{q_n}(1+\frac{a_{\max}}{a_{\min}})\left(r_n\,\|g_N(\cdot,\xi)\|_{H^1(\cD)} + \|g_n\|_{H^1(\cD)}\right)
  \end{align*}
   for almost all $\xi\in \Gamma$ with $r_n = \frac{\|a_n\|_{L^\infty( \cD)}}{a_{\min}}$.
\end{lemma}
\begin{proof}
  Introducing the decomposition $v = v_0 + g_N$ with $v_0\in L^2(\Gamma; H^1_0(\cD))$ the equation \eqref{eq:general_bvp} is equivalent to
  \begin{equation}\label{eq:def_v0}
    \int_\cD a_N\nabla v_0 \cdot \nabla w\,\d\bx = \int_\cD \left(F\,w - a_N\nabla g_N\cdot \nabla w\right)\d \bx
  \end{equation}
  for all $w\in H_0^1(\cD)$ and almost everywhere in $\Gamma$.  
  We choose the test function $w = v_0$ and obtain with the Cauchy-Schwarz and the Poincar\'e inequality with constant $C_D>0$
  \begin{equation}\label{eq:reg_estimate_v0}
    |v_0|_{H^1(\cD)}
    \le \frac{C_D}{a_{\min}} \|F\|_{L^2(\cD)} + \frac{a_{\max}}{a_{\min}} \,|g_N|_{H^1(\cD)}.
  \end{equation}
  Taking now derivatives in \eqref{eq:def_v0} with respect to $\xi_n$ gives
  \begin{equation}\label{eq:def_v0_dxi}
    \int_{\cD} \left( a_n\nabla v_0\cdot \nabla w + a_N\nabla (\partial_{\xi_n} v_0)\cdot\nabla w\right) \d \bx
    =
    \int_{\cD}\left(\partial_{\xi_n} F\,w - a_n\nabla g_N\cdot\nabla w - a_N\nabla g_n\cdot\nabla w\right)\d \bx
  \end{equation}
  and choosing $w=\partial_{\xi_n} v_0$ together with \eqref{eq:reg_estimate_v0} yields
  \begin{align*}
    |\partial_{\xi_n} v_0|_{H^1(\cD)}
    &\le
      \frac{C_D}{a_{\min}} \left(\|\partial_{\xi_n} F\|_{L^2(\cD)} + r_n\,\|F\|_{L^2(\cD)}\right) \\
    & \quad + r_n\,(1+\frac{a_{\max}}{a_{\min}})\, |g_N|_{H^1(\cD)}
      + \frac{a_{\max}}{a_{\min}}\,|g_n|_{H^1(\cD)}.
  \end{align*}
  Taking again derivatives in \eqref{eq:def_v0_dxi} with respect to $\xi_n$ and using $\partial_{\xi_n}^2 g_N=0$ produce
  \begin{equation*}
     \int_{\cD} \left(2\,a_n\,\nabla(\partial_{\xi_n}v_0)\cdot \nabla w + a_N\nabla(\partial_{\xi_n}^2 v_0)\cdot \nabla w\right)\d\bx 
    = \int_{\cD} \left(\partial_{\xi_n}^2 F\,w - 2\,a_n\nabla(\partial_{\xi_n} g_N)\cdot\nabla w\right)\d \bx
  \end{equation*}
 and $w=\partial_{\xi_n}^2 v_0$ gives the desired estimate for the case $q_n=1$.
    
  The desired result follows by repeated application of the above arguments and the triangle inequality, taking into account $\partial_{\xi_n}^{q_n+1} g_N = 0$ for $q_n\ge 2$.
\end{proof}

The application of the regularity estimate from Lemma~\ref{lem:regularity_derivatives_xi} to the state equation presented in equation \eqref{eq:opt_pa} yields the following result.

\begin{corollary}
  \label{cor:state_regularity_wrt_stochastic}
  Let $y\in L^2(\Gamma;H^1(\cD))$ be the solution of \eqref{eq:opt_pa} with $u\in H^{1/2}(\partial\cD)$. Then, there holds
  \begin{align*}
    \frac{\|\partial_{\xi_n}^{q_n+1} y\|_{L^2(\Gamma; H^1(\cD))}}{(q_n+1)!}
    &\le C\,r_n^{q_n+1}\left(\|f_N\|_{L^2(\Gamma; L^2(\cD))} + \|u+b_N\|_{L^2(\Gamma; H^{1/2}(\partial\cD))}\right) \\
    &\quad + C\,r_n^{q_n}\left(\|f_n\|_{L^2(\Gamma; L^2(\cD))} + \|b_n\|_{L^2(\Gamma; H^{1/2}(\partial\cD))}\right) \\
    &\le C(u,a_N,f_N,b_N)\,r_n^{q_n+1}.
  \end{align*}
\end{corollary}

\subsection{Stochastic Galerkin method}

To obtain a fully discrete scheme, we introduce finite element spaces on the parametric domain $\Gamma$ and the physical domain $\mathcal{D}$. We first decompose the parametric domain denoted by $\Gamma \subset \mathbb{R}^N$ into a finite number of disjoint $\R^N$--boxes $B_j^{N}$ as follows
\[
 \Gamma =  \bigcup \limits_{j \in J} B_j^{N} = \bigcup \limits_{j \in J} \prod \limits_{n=1}^{N} (r_j^{n},s_j^{n}),
\]
where $B_j^{N} \cap B_i^{N} = \emptyset$ for $j \neq i \in J$  and     $(r_j^{n},s_j^{n}) \subset \Gamma_n$ for $n=1,\ldots,N$ and $ j \in J$. On each $\Gamma_n$, the mesh size is denoted by 
\[
\gamma_n = \max \limits_{j \in J} \lvert s_j^{n} - r_j^{n}\rvert \quad \hbox{for} \;\; 1 \leq n \leq N
\]
and the maximum mesh size of $\Gamma$ becomes   $\gamma = \max \limits_{n=1,\ldots, N} \gamma_n$. The stochastic (parametric) finite element space is given by
\begin{equation}
\mathcal S^{\gamma} = \{\psi^{\gamma} \in L^2(\Gamma): \, \psi|_{B_j^N} \in \text{span}\big( \prod \limits_{n=1}^{N} \xi_n^{\ell_n} : \; \ell_n \in \mathbb{N}, \, \ell_n \leq q_n, \, n=1, \ldots, N  \big)\},
\end{equation}
which  has at most $q_n$ degree on each direction $\xi_n$ for a multi--index $q = (q_1,\ldots,q_N)$. 



Next, we consider finite element subspaces on the physical domain $\mathcal{D}$. Let  $\{ \mathcal{T}_h\}_h$ be a quasi-uniform family of triangulations of $\mathcal{D}$  with $\overline{\mathcal{D}} = \bigcup_{K \in \mathcal{T}_h} \overline{K}$ and maximum mesh size $h=\max \limits_{K \in \mathcal{T}_h} h_{K}$, where $h_{K}$ is the diameter of an element $K$.  Furthermore, $\mathcal{E}_h$ denotes the partition of the boundary $\partial\mathcal D$ induced by $\mathcal T_h$, meaning that each $E\in \mathcal E_h$ with $E\subset\partial\mathcal D$ is an edge of some triangle $ K \in \mathcal T_h$. Denoting the set of  all linear polynomials on $K$ by $\mathbb{P}(K)$,  the finite element spaces of state, test and control variables are defined by
\begin{align*}
  Y_h  &= \{y_h \in C(\mathcal{D}) : \,  y_h|_K\in \mathbb{P}(K) \quad \forall K \in \mathcal{T}_h\},
  \\
  V_h  &= \{v_h \in Y_h : \,  v_h|_{\partial \cD} = 0\}, \\
  U_h  &= \{u_h \in C(\partial\mathcal{D}) : \,  u_h|_E\in \mathbb{P}(E) \quad \forall E \in \mathcal{E}_h\}.
\end{align*}
To represent functions from $Y_h$ we introduce the nodal basis $\{\phi_i\}\subset Y_h$ satisfying  $\phi_m(x_n)=\delta_{mn}$ for all $m,n=1,\ldots,\N:=\text{dim}(Y_h)$, where $x_n$ are the nodes of the mesh $\mathcal{T}_h$. We number the nodes in such a way that $x_n$ are interior nodes for $n=1,\ldots,\N_{\textup{I}}$ and boundary nodes for $n=\N_{\textup{I}}+1,\ldots,\N=\N_{\textup{I}}+\N_{\textup{B}}$.


Now, we are ready to state the tensor product finite element spaces on $\mathcal{D} \otimes \Gamma $. As test space we define
\begin{equation*}
  \mathcal{V}_h^{\gamma} = V_h \otimes \mathcal{S}^{\gamma} = \text{span}\{
  \phi_h \, \psi^{\gamma}: \, \phi_h \in V_h, \, \psi^{\gamma} \in \mathcal{S}^{\gamma}
  \}
\end{equation*}
and the trial space is defined by
\begin{equation*}
	\mathcal{Y}_h^\gamma	 := Y_h\otimes \mathcal S^\gamma.
\end{equation*}

With the spaces defined above we can now state the following fully discrete optimal control problem as approximation of \eqref{eqn:objec}--\eqref{eqn:constPDE}:
\begin{subequations}
  \label{eq:discrete_ocp}
  \begin{equation}\label{eqn:objec_disc}
    \min \limits_{(y_h^\gamma,u_h) \in \mathcal{Y}_h^\gamma\times U^{\textup{ad}}_h} \; \mathcal{J}(y_h^\gamma,u_h) = \frac{1}{2} \int_{\Gamma} \int_{\mathcal{D}} \left( y_h^\gamma -y^d \right)^2 \d\boldsymbol{x} \, \rho(\xi)\, \d\xi + \frac{\alpha}{2}  \int_{\partial\mathcal{D}} u_h^2 \, \d\boldsymbol{x}
  \end{equation}
  subject to
  \begin{equation}\label{eqn:vari_disc}
    \begin{aligned}
      y_h^{\gamma}|_{\partial\cD}&= u_h + \Pi_h(b_{N}) \\
      a[y_h^{\gamma},v_h^\gamma]_{\rho} &= [f_N,v_h^\gamma]_{\rho}, \quad \forall v_h^\gamma \in \mathcal{V}_h^{\gamma},
    \end{aligned}
  \end{equation}
  and
  \begin{equation}\label{defn:add_dis}
    u_h\in U_h^{\textup{ad}} :=  \{ u_h \in U_h \colon u_a \leq  u_h(\boldsymbol{x}) \leq u_b,\ \text{f.a.a.}\ \boldsymbol{x} \in \partial \cD \}.
  \end{equation}
\end{subequations}
Here, $\Pi_h$, cf. \eqref{eq:l2_proj_u}, is a projection operator onto $U_h$.

A  pair $(y_h^{\gamma},u_h) \in \mathcal{Y}_h^{\gamma} \times U_h^{\textup{ad}}$ is a unique solution of the control problem \eqref{eq:discrete_ocp} if and only if  an adjoint state $p_h^{\gamma} \in \mathcal{V}_h^{\gamma}$ exists such that the triplet $(y_h^{\gamma},u_h,p_h^{\gamma}) \in \mathcal{Y}_h^{\gamma} \times U_h^{\textup{ad}} \times \mathcal{V}_h^{\gamma} $ solves the following discrete optimality system
\begin{subequations}\label{eq:opt_ph}
  \begin{alignat}{2}
    y_h^{\gamma}|_{\partial\cD}= u_h + \Pi_h(b_{N}),\quad a[y_h^{\gamma},v_h^\gamma]_{\rho} &= [f_N, v_h^\gamma]_{\rho} &\qquad &\forall v_h^\gamma \in \mathcal{V}_h^{\gamma},
    \label{eq:state_equation_discrete}\\
    a[q_h^\gamma,p_h^{\gamma}]_{\rho}  &= [y_h^{\gamma}-y^d,q_h^\gamma]_{\rho} &&\forall q_h^\gamma \in \mathcal{V}_h^{\gamma},
    \label{eq:adjoint_equation_discrete}\\
    (\alpha\,u_h - \E[g_h], w_h-u_h)_{\partial \cD} &\geq 0 &&\forall  w_h \in U^{\textup{ad}}_h.\label{eq:opt_cond_discrete} 
  \end{alignat}  
\end{subequations}
Here, $\E[g_h]\in U_h$ is also a variational co-normal derivative defined by
\[
(g_h, w_h)_{\partial\cD} = a[S_h^{\gamma}(w_h), p_h^\gamma]_\rho - [y_h^\gamma-y^d, S_h^{\gamma}(w_h)]_\rho \qquad \forall w_h\in U_h,
\]
with $S_h^{\gamma} \colon  U_h \to \mathcal{Y}_h^{\gamma}$ some arbitrary discrete extension operator.

Last, we note that the discrete variational inequality \eqref{eq:opt_cond_discrete}  can also be written as
\begin{equation}\label{eq:variationalform_full}
  j_h^\gamma{'}(u_h) (w_h - u_h) = (\alpha \, u_h, w_h-u_h)_{\partial \cD} + [y_h^{\gamma}-y^d, y_h^{\gamma} (w_h) - y_h^{\gamma}]_{\rho}  \geq 0, \quad \forall w_h \in U^{\textup{ad}}_h,
\end{equation}
where $j_h^\gamma(u_h) = \mathcal J(y_h^\gamma (u_h), u_h)$ and 
$y_h^\gamma(w_h)$ is the solution of the problem \eqref{eq:state_equation_discrete} with boundary data $w_h + \Pi_h(b_{N})$. 



\section{Error estimates}\label{sec:error_estimates} 

In this section, we derive estimates for the approximation error of the solution of the discretized problem  \eqref{eq:opt_ph}. Before the derivation of the estimates,  we introduce some projection operators into the approximation spaces. First, consider the $L^2$-projection operator in the parametric domain  $\Pi_{\gamma}\colon L^2(\Gamma) \rightarrow \mathcal{S}^{\gamma}$  
\begin{equation}
  (\Pi_{\gamma}(\xi) - \xi, \zeta)_{L^2(\Gamma)} = 0  \qquad \forall  \zeta \in \mathcal{S}^{\gamma}.
\end{equation}
From \cite[Equation (3.6)]{IBabuska_RTempone_GEZouraris_2004a}, for sufficiently smooth $w \in L^2(\Gamma)$ we have
\begin{eqnarray}\label{est:parametric}
  \|w - \Pi_\gamma(w)\|_{L^2(B_j^N)} \leq
  \sum_{n=1}^N \left(\frac{\gamma_n}{2}\right)^{q_n+1} \frac{\|\partial_{\xi_n}^{q_n+1} w\|_{L^2(B_j^N)}}{(q_n+1)!}
\end{eqnarray}
for all $j\in J$.

Furthermore, the $L^2$-projection operator on the boundary of the physical domain  $\Pi_h\colon L^2(\partial \cD) \rightarrow U_h$ 
\begin{equation}\label{eq:l2_proj_u}
    (\Pi_h(w)   - w, \vartheta)_{L^2(\partial \cD)} = 0  \qquad \forall  \vartheta \in U_h,
\end{equation}
satisfies (see \cite{PGCiarlet_1978a} and \cite{DAFrench_JTKing_1993})
\begin{subequations}\label{ineq:l2_projection}
    \begin{align}
        \|w -\Pi_h(w)\|_{L^2(\partial \cD)} &\leq C \, h^{s} \|w\|_{H^s(\partial \cD)} && \hbox{for}\ w \in H^s(\partial \cD), \; 0 \leq s \leq 2, \label{ineq:l2_projectiona} \\
        \|w -\Pi_h(w)\|_{L^2(\partial \cD)} &\leq C \, h^{s-1/2} \|w\|_{H^s(\cD)} && \hbox{for}\ w \in H^s(\cD), \; 1/2 \leq s \leq 2, \label{ineq:l2_projectionb} \\ 
        \|(I - \Pi_h)(a_N\,\partial_n w) \|_{L^2(\partial \cD)} &\leq C \, h^{1/2} \|w\|_{H^2(\cD)} && \hbox{for}\ w \in H^2(\cD).  \label{ineq:l2_projectionc}
    \end{align}
\end{subequations}
The Ritz projection denoted by $R_h: H^1_0(\cD) \rightarrow V_h$  (see, e.g., \cite[Sec. 3.1]{PGCiarlet_1978a}) is defined by
\begin{equation}\label{ritz_def}
 a[R_h(q), v_h]_{\rho} = a[q, v_h]_{\rho}, \quad  q \in H^1_0(\cD), \quad \forall   v_h \in V_h
\end{equation}
and satisfies $\| \nabla R_h(q)\|_{L^2(\cD)} \leq \| \nabla q\|_{L^2(\cD)}$ as well as
\begin{equation}\label{ritz_ineq}
 \| q -R_h(q)\|_{H^s(\cD)} \leq C \, h^{r -s} \|q\|_{H^r(\cD)} \quad  \hbox{for} \;\; q \in H^1_0(\cD) \cap H^r(\cD), \quad  0 \leq s \leq 1 \leq r \leq 2.
\end{equation}



\noindent Further, we introduce a discrete control-to-state map $U^{\textup{ad}} \ni w \mapsto y_h^{\gamma}(w) \in \mathcal{Y}_h^{\gamma}$ defined by
\begin{equation}\label{eqn:aux2}
  y_h^{\gamma}(w) |_{\partial\cD}= \Pi_{h}(w+b_{N}),\quad a[y_h^{\gamma}(w), v_h^\gamma]_{\rho} = [f_N,v_h^\gamma]_{\rho} \quad \forall v_h^\gamma \in \mathcal{V}_h^{\gamma}.
\end{equation}
Then, from \cite[Lemma~5.1]{ECasas_JPRaymond_2006b} and \cite[Theorem~2.3]{ECasas_JPRaymond_2006a} taking into account the linearity of the problem (considering the problem with either $f_N \equiv 0$ and $b_N \equiv 0$ or $u \equiv 0$), and the stability of $\Pi_h$ in $H^s(\partial \cD)$ for $0 \leq s \leq 1$  \cite{ECasas_JPRaymond_2006a},  the following inequality holds 
\begin{align}\label{ineq:yh_Piu}
  \|y_h^{\gamma}(w) \|_{\yH} &\le C \left( \|\Pi_h(w + b_{N})\|_{\uH} + \|f_N\|_{\yL}  \right) \nonumber \\
  &\le C \left( \|w + b_{N}\|_{\uH} + \|f_N\|_{\yL}  \right)
\end{align}
and the well-known inverse estimate \cite{MBerggren_2004a} yields
\begin{equation}\label{ineq:yh_Piu_inv}
    \|y_h^{\gamma}(w) \|_{\yH}  \leq C \left( h^{-1/2} \|\Pi_h(w + b_{N}) \|_{\uL}  + \|f_N\|_{\yL}  \right).
\end{equation}

For technical reasons we need to define some discrete approximation of the exact control, belonging to $U_h^{\textup{ad}}$ that can be used as test function in \eqref{eq:variationalform_full}. Note that $\Pi_h(u)$ does not necessarily belong to the discrete admissible set. We follow the construction of \cite{TApel_MMateos_JPfefferer_ARosch_2018,ECasas_JPRaymond_2006b} and define 
\begin{equation*}
	I_n = \int_{\partial\cD} d(\bx)\,\phi_{\mathcal N_{\textup{I}}+n}(\bx) \, \d s_{\bx},\quad n=1,\ldots,\mathcal N_{\textup B}.
\end{equation*}
for  $d = \alpha\,u - \E[a_N\,\partial_n p]$. Then, we define the desired test function by 
\begin{equation}\label{eq:def_uhstar}
	u_h^*(\bx) = \sum_{n=1}^{\mathcal N_{\textup{B}}} u_n^*\,\phi_{\mathcal N_{\textup{I}}+n}(\bx),
\end{equation}
where the coefficients are given by
\begin{equation*}
	u_n^* := \begin{cases}
	      \displaystyle\frac1{I_n} \int_{\partial\cD} d(\bx)\,u(\bx)\,\phi_{\mathcal N_{\textup{I}}+n}(\bx)\,\d s_{\bx}, &\quad \text{if}\ I_n\ne 0,\\
	      \displaystyle\frac{1}{|S_n|} \int_{S_n} u(\bx)\,\d s_{\bx}, &\quad\text{if}\ I_n=0,
	\end{cases}
\end{equation*}
with $S_n = \text{supp}(\phi_{\mathcal N_{\textup{I}}+n}|_{\partial\cD})$. The important properties following from this construction and which we need in the following are that $u_h^*$ belongs to $U_h^{\textup{ad}}$ and that $u_h^* = u$ on the active set. Both properties together also imply
\begin{equation}\label{eq:uhstar_in_varinequ}
	(\alpha\,u - \E[a_N\,\partial_n p], u_h^* - u)_{\partial\cD} = 0.
\end{equation}

To establish the main result of this section we first show a general a priori error estimate for the optimal control problem \eqref{eq:opt_p}, discretized by the stochastic Galerkin method.
\begin{theorem}
\label{thm:general_estimate}
 Let $(y,u,p) \in L^2(\Gamma;H^1(\cD)) \times U^{\textup{ad}} \times L^2(\Gamma;H^1_0(\cD) \cap H^2(\cD))$  and $(y_h,u_h,p_h) \in \mathcal{Y}_h^{\gamma} \times U_h^{\textup{ad}} \times  \mathcal{V}_h^{\gamma}$ be the solutions of \eqref{eq:opt_p} and \eqref{eq:opt_ph}, respectively. Then, we have the following a priori error estimates 
 \begin{align}\label{eq:basic_error_estimate}
  & \sqrt\alpha\|u-u_h\|_{L^2(\partial \cD)} + \|y - y_h^\gamma \|_{\yL} \nonumber \\
  & \le C\left( \|u - u_h^*\|_{L^2(\partial \cD)}
  + \|y-y_h^\gamma(u)\|_{L^2(\Gamma;L^2(\cD))}
  +\sup_{w_h\in U_h \setminus\{0\}} \frac{|a[S_h^\gamma (w_h), p]_{\rho}|}{\|w_h\|_{L^2(\partial \cD)}} \right),
 \end{align}
 where the generic constant $C$ is independent of $h, \gamma$ and $q = (q_1,\ldots,q_N)$, and the extension operator $S_h^\gamma\colon U_h \to \mathcal Y_h^\gamma$ is defined by $\widetilde y_h^\gamma = S_h^\gamma(w_h)$,
\begin{equation*}
	\widetilde y_h^\gamma|_{\partial\cD} = w_h,\qquad a[\widetilde y_h^\gamma, v_h^\gamma]_\rho = 0 \quad \forall  v_h^\gamma \in \mathcal V_h^\gamma.
\end{equation*}
\end{theorem}
\begin{proof}
Setting $w=u_h \in U^{\textup{ad}}$ and $w_h= u_h^* \in U_h^{\textup{ad}}$ in \eqref{eq:variationalform_parametric} and \eqref{eq:variationalform_full}, respectively, we obtain
\begin{align}\label{pr1}
&\alpha\,\|u-u_h\|_{L^2(\partial \cD)}^2 \nonumber \\
& \quad =  \alpha\,(u, u-u_h)_{\partial\cD}
- \alpha\,(u_h, u_h^*-u_h)_{\partial\cD}
- \alpha\,(u_h,u-u_h^*)_{\partial\cD}
\nonumber \\
& \quad \le  [y- y^d, y(u_h) - y]_\rho
            +  [y_h^{\gamma} - y^d, y_h^{\gamma}(u_h^*) -  y_h^{\gamma}]_\rho - \alpha (u_h,u-u_h^*)_{\partial\cD},
\end{align}
where $y(u_h) \in L^2(\Gamma;L^2(\cD))$ with $y(u_h)|_{\partial \cD} =  b_N + u_h$ solves the following auxiliary problem 
\begin{align}\label{eqn:aux1}
  & \int_\Gamma \int_\cD - y(u_h)  \, \nabla \cdot (a_N \nabla v)   \, \d \bx\,\rho\, \d\xi + \int_\Gamma \int_{\partial\cD} (b_N + u_h)\, a_N \,  \partial_n v\,\d s_\bx
  \, \rho\,\d\xi \nonumber \\ 
  &\quad  = \int_\Gamma \int_\cD f_N \,v \,\d \bx \, \rho\, \d\xi \qquad\qquad \forall  v \in L^2(\Gamma; H_0^1(\cD) \cap H^2(\cD)).
\end{align}
Rearranging the terms in \eqref{pr1} and using Young's inequality give
\begin{align}\label{pr2}
\alpha\|u-u_h\|_{L^2(\partial \cD)}^2
& \le [y - y^d, y(u_h) - y]_{\rho} + [ y_h^{\gamma} - y, y -y_h^{\gamma}]_{\rho} \nonumber \\
& \quad + [y_h^{\gamma} - y, y_h^{\gamma}(u_h^*) - y]_{\rho} + [y - y^d, y_h^{\gamma}(u_h^*) - y_h^{\gamma}]_{\rho}
- \alpha (u_h,u-u_h^*)_{\partial\cD}\nonumber \\
&  =  [y - y^d, y(u_h) - y -  ( y_h^{\gamma} - y_h^{\gamma}(u_h^*) ) ]_{\rho} - \|y - y_h^{\gamma}\|_{\yL}^2\nonumber \\
& \quad + [y_h^{\gamma} - y, y_h^{\gamma}(u_h^*) - y]_{\rho} - \alpha (u_h,u-u_h^*)_{\partial\cD} \nonumber \\
&  \leq [y - y^d, y(u_h) - y -  ( y_h^{\gamma} - y_h^{\gamma}(u_h^*) ) ]_{\rho} - \frac12 \|y - y_h^{\gamma}\|_{\yL}^2\nonumber \\
& \quad  +  \frac12\,\|y - y_h^{\gamma}(u_h^*)\|_{\yL}^2 - \alpha (u_h,u-u_h^*)_{\partial\cD}.
\end{align}
We discuss the four terms on the right-hand side separately. For the first term, we use the strong formulation of the adjoint equation, the formulations \eqref{eqn:state_weak} and \eqref{eqn:aux1} with $v=p$,  and the integration-by-parts formula to deduce
\begin{align}\label{eq:pr3}
	&[y - y^d, y(u_h) - y -  ( y_h^{\gamma} - y_h^{\gamma}(u_h^*) ) ]_{\rho}	\nonumber\\
	&\quad = - (\E[a_N \partial_n p], u_h - u)_{\partial\cD} + a[y_h^\gamma(u_h^*) - y_h^\gamma, p]_{\rho}
	- (\E[a_N \partial_n p], u_h^* - u_h)_{\partial\cD} \nonumber\\
	&\quad = a[y_h^\gamma(u_h^*) - y_h^\gamma, p]_{\rho} - (\E[a_N\partial_n p], u_h^* - u)_{\partial\cD}.
\end{align}
The second term is kicked to the left-hand side.
For the third term in \eqref{pr2} we get with the triangle inequality and a Lipschitz-estimate for $y_h^\gamma$ \cite[Thm.~2.2]{ECasas_JPRaymond_2006b}
\begin{align}\label{eq:pr4}
	\|y-y_h^\gamma(u_h^*)\|_{L^2(\Gamma;L^2(\cD))}^2
	&\le 2\,\|y-y_h^\gamma(u)\|_{L^2(\Gamma;L^2(\cD))}^2 + 2\,\|y_h^\gamma(u) - y_h^\gamma(u_h^*)\|_{L^2(\Gamma;L^2(\cD))}^2 \nonumber\\
	&\le 2\,\|y-y_h^\gamma(u)\|_{L^2(\Gamma;L^2(\cD))}^2 + C\,\|u - u_h^*\|_{L^2(\partial\cD)}^2.	
\end{align}
The fourth term in \eqref{pr2} gives
\begin{align}\label{eq:pr5}
	&-\alpha(u_h, u-u_h^*)_{\partial\cD} = \alpha(u - u_h, u-u_h^*)_{\partial\cD} + \alpha(u, u_h^* - u)_{\partial\cD} \nonumber \\
	&\quad\le \frac\alpha2\,\|u-u_h\|_{L^2(\partial\cD)}^2
	+ \frac\alpha2\,\|u-u_h^*\|_{L^2(\partial\cD)}^2
	+ \alpha\,(u, u_h^*-u)_{\partial\cD}.
\end{align}
Insertion of \eqref{eq:pr3}, \eqref{eq:pr4}, and \eqref{eq:pr5} into \eqref{pr2} then leads to the estimate
\begin{align}\label{eq:pr6}
	&\alpha\|u-u_h\|_{L^2(\partial \cD)}^2 + \|y - y_h^{\gamma}\|_{\yL}^2 \nonumber\\
	&\quad \le C\,\Big(\|u - u_h^*\|_{L^2(\partial\cD)}^2 + \|y-y_h^\gamma(u)\|_{L^2(\Gamma;L^2(\cD))}^2 \nonumber \\
	&\qquad \qquad + a[y_h^\gamma(u_h^*) - y_h^\gamma, p]_{\rho} + (\alpha\,u - \E[a_N\,\partial_n p], u_h^* - u)_{\partial\cD}\Big).
\end{align}
The last term on the right-hand side vanishes due to the definition of the function $u_h^*$, see \eqref{eq:uhstar_in_varinequ}. Furthermore, with the definition of $S_h^\gamma$ and the relation $y_h^\gamma(u_h^*) - y_h^\gamma = S_h^\gamma(u_h^* - u_h)$ we may rewrite the third term in \eqref{eq:pr6} as
\begin{align*}
	& a[y_h^\gamma(u_h^*) - y_h^\gamma, p]_{\rho} \nonumber \\
    & \quad =	\sup_{w_h\in U_h \setminus\{0\}} \frac{|a[S_h^\gamma(w_h), p]_{\rho}|}{\|w_h\|_{L^2(\partial \cD)}}
	\left(\|u-u_h^*\|_{L^2(\partial \cD)} + \|u - u_h\|_{L^2(\partial \cD)}\right).
\end{align*}
Together with Young's inequality we arrive at the desired estimate.

\end{proof}

Next, we derive upper bounds for the three terms on the right-hand side of \eqref{eq:basic_error_estimate}.
\begin{lemma}\label{lem:projection_estimate}
	Let $y^d\in H^t(\cD)$ for some arbitrary $t>0$. The projection $u_h^*\in U_h^{\textup{ad}}$ from \eqref{eq:def_uhstar} fulfills the estimate
	\begin{equation}\label{eq:estimate_control_projection}
		\|u-u_h^*\|_{L^2(\partial\cD)} \le C \,  h^{1/2}.
                  \end{equation}
\end{lemma}
\begin{proof}
  The assumption on the desired state guarantees $u\in H^{1/2+\varepsilon}(\partial\cD)$ for some sufficiently small $\varepsilon>0$, which can be deduced with bootstrapping arguments like in \eqref{eq:bootstrapping} taking into account the regularity estimates from \cite[Corollary 5.3]{TApel_MMateos_JPfefferer_ARosch_2015}. The estimate \eqref{eq:estimate_control_projection} is then a consequence of \cite[Lemma 5.6]{TApel_MMateos_JPfefferer_ARosch_2018}.
\end{proof}
\begin{lemma}\label{lem:FE_error_state_eq}
  Suppose $u\in H^{1/2}(\partial\cD)$ and let $y$ solve \eqref{eq:opt_pa}. For the approximation $y_h^\gamma(u)$ from \eqref{eqn:aux2} there holds
  \begin{equation*}
    \|y-y_h^\gamma(u)\|_{L^2(\Gamma; L^2(\Omega))} 
    \le C\left(
      h 
      + \sqrt{\sum_{n=1}^N(q_n+2)\,\left(\frac{\gamma_n\,r_n}{2}\right)^{q_n+1}}
    \right),
  \end{equation*}
  with some $C>0$ depending only on the data $a_N, f_N, b_N, u$.
\end{lemma}
\begin{proof}
The proof uses a duality argument. To this end, let $z\in L^2(\Gamma;H_0^1(\cD) \cap H^2(\cD))$ be the strong solution of
\begin{equation}
\begin{aligned} \label{eq:aux_adj}
	-\nabla\cdot(a_N \,\nabla z) &= y-y_h^\gamma(u) &\quad&\text{in}\ \Gamma\times \cD,\\
	z &= 0 && \text{on}\ \Gamma\times \partial\cD.
\end{aligned}
\end{equation}
The function $z$ fulfills the \emph{a priori} estimate
\begin{equation}
\label{eq:dual_prob_regularity}
	\|z(\cdot,\xi)\|_{H^2(\cD)} \le C\,\|(y-y_h^\gamma(u))(\cdot,\xi)\|_{L^2(\cD)}
\end{equation}
for almost all $\xi\in \Gamma$.
By using the definitions of $z$ from \eqref{eq:aux_adj} and $y$ from \eqref{eq:opt_pa}, the integration-by-parts formula, the fact that $b_N = \Pi_\gamma(b_N)$, and  the Galerkin orthogonality, we get for arbitrary $z_h^\gamma\in \mathcal V_h^\gamma$ 
\begin{align}\label{eq:FE_error_h1_splitting}
  &\|  y - y_h^{\gamma}(u) \|_{\yL}^2  \nonumber \\
  &\quad =a[y-y_h^\gamma(u), z-z_h^\gamma]_{\rho} - [a_N\partial_n z, u+b_N - \Pi_h(u+b_N)]_{\rho,\partial\cD}
    = E_1 + E_2.
\end{align}
To bound the term $E_2$, we use the orthogonality of $\Pi_h$, the Cauchy-Schwarz inequality, the inequalities \eqref{ineq:l2_projectiona} and \eqref{ineq:l2_projectionc}, and the regularity estimate \eqref{eq:dual_prob_regularity}, and infer
\begin{align}
\label{pr9}
  E_2 &= [u+b_N - \Pi_h(u+b_N), a_N\,\partial_n z - \Pi_h(a_N\,\partial_n z)]_{\rho,\partial\cD}
        \nonumber \\
      &\le C\,h\,\|u+b_N\|_{L^2(\Gamma;H^{1/2}(\partial\cD))}\, \|z\|_{L^2(\Gamma; H^2(\cD))}
        \nonumber \\
      &\le
        C\,h\,(\|u\|_{H^{1/2}(\partial\cD)} + \|b_N\|_{L^2(\Gamma; H^{1/2}(\partial\cD))})\,\|y-y_h^\gamma(u)\|_{L^2(\Gamma; L^2(\cD))}.
\end{align}

To deduce an estimate for the term $E_1$ on the right-hand side of \eqref{eq:FE_error_h1_splitting} 
we choose $z_h^\gamma = R_h(\Pi_\gamma(z))$ with $R_h$ defined in \eqref{ritz_def}, exploit the stability of $R_h$ in $H^1(\cD)$, and the estimates \eqref{ritz_ineq} and \eqref{est:parametric} to arrive at
\begin{equation}
  \label{eq:error_est_dual_solution}
  \|z-z_h^\gamma\|_{L^2(B_j^N; H^1(\cD))}
  \le C\left(h\,|z|_{L^2(B_j^N; H^2(\cD))} + \sum_{n=1}^N \left(\frac{\gamma_n}{2}\right)^{q_n+1} \frac{\|\partial_{\xi_n}^{q_n+1} z\|_{L^2(B_j^N; H^1(\cD))}}{(q_n+1)!}\right)
\end{equation}
for all $j\in J$.
Applying Lemma~\ref{lem:regularity_derivatives_xi} leads to
  \begin{equation*}
    \frac{\|\partial_{\xi_n}^{q_n+1} z\|_{L^2(B_j^N; H^1(\cD))}}{(q_n+1)!}
    \le
    C\,\sum_{i=0}^{q_n+1} \frac{1}{i!}\,r_n^{q_n+1-i}\,\|\partial_{\xi_n}^i (y-y_h^\gamma(u))\|_{L^2(B_j^N; L^2(\cD))}.
  \end{equation*}
  A further application of Corollary~\ref{cor:state_regularity_wrt_stochastic} gives  
  \begin{align*}
    \frac1{i!} \|\partial_{\xi_n}^i y\|_{L^2(B_j^N; L^2(\cD))}
    &\le
      C\,r_n^i\, \left(\|f_N\|_{L^2(B_j^N; L^2(\cD))} + \|u + b_N\|_{L^2(B_j^N; H^{1/2}(\partial\cD))}\right) \\
    &\quad +   C\,r_n^{i-1}\,\left(\|f_n\|_{L^2(B_j^N; L^2(\cD))} +
      \|b_n\|_{L^2(B_j^N; H^{1/2}(\partial\cD))}\right).
  \end{align*}
  With an inverse inequality we deduce furthermore
  \begin{align*}
    \frac1{i!} \|\partial_{\xi_n}^i y_h^\gamma(u)\|_{L^2(B_j^N; L^2(\cD))} &\le C\,\frac{1}{i!}\,(\gamma_n\,r_n)^{-i}\,r_n^i \,\|y_h^\gamma(u)\|_{L^2(B_j^N; L^2(\cD))} \\
    &\le C\,r_n^i\,\|y_h^\gamma(u)\|_{L^2(B_j^N; L^2(\cD))}
  \end{align*}
  and the previous three estimates then give
  \begin{align*}
    \frac{\|\partial_{\xi_n}^{q_n+1} z\|_{L^2(B_j^N; H^1(\cD))}}{(q_n+1)!}
    &\le C\,
    (q_n+2)\,r_n^{q_n+1}
    \Big(\|f_N\|_{L^2(B_j^N; L^2(\cD))} + \|u + b_N\|_{L^2(B_j^N;H^{1/2}(\partial\cD))} \\
    & \quad + \|f_n\|_{L^2(B_j^N; L^2(\cD))} + \|b_n\|_{L^2(B_j^N; H^{1/2}(\partial\cD))}
    + \|y_h^\gamma(u)\|_{L^2(B_j^N; L^2(\cD))}\Big).
  \end{align*}
  Insertion into \eqref{eq:error_est_dual_solution}, summation over all the boxes $B_j^N$, $j\in J$, and an application of \eqref{eq:dual_prob_regularity} then yield the following error estimate for the approximation of the dual solution
  \begin{equation}\label{eq:final_est_dual_solution}
    \|z-z_h^\gamma\|_{L^2(\Gamma; H^1(\cD))}
    \le C\left(h\,\|y-y_h^\gamma(u)\|_{L^2(\Gamma; L^2(\cD))}
    + \sum_{n=1}^N (q_n+2)\,\left(\frac{\gamma_n\,r_n}{2}\right)^{q_n+1}\,D\right)
  \end{equation}
  with
  \begin{align*}
    D&:= \|f_N\|_{L^2(\Gamma; L^2(\cD))} + \|u + b_N\|_{L^2(\Gamma;H^{1/2}(\partial\cD))} \\
    &\quad+ \|f_n\|_{L^2(\Gamma; L^2(\cD))} + \|b_n\|_{L^2(\Gamma; H^{1/2}(\partial\cD))}
    + \|y_h^\gamma(u)\|_{L^2(\Gamma; L^2(\cD))}.
  \end{align*}
  The last term in $D$ can be omitted as it is bounded by the data $f_N, f_n, u+b_N$, and $b_n$ as well, see \eqref{ineq:yh_Piu}. In the following we will hide the dependencies on the data in the generic constant $C$.

  Consider now the term $E_1$ in \eqref{eq:FE_error_h1_splitting}.
  An application of  the Cauchy-Schwarz inequality, the bound of $a_N$ in \eqref{assumption_truncatedD},  the triangle inequality, and the inequalities in \eqref{ineq:yh_Piu} and \eqref{eq:final_est_dual_solution} result in 
\begin{align}
\label{pr10}
  E_1
  & \leq  a_{\max} \, \big( \|y\|_{\yH} + \|y_h^{\gamma}(u)\|_{\yH} \big)   \|z - z_h^\gamma\|_{L^2(\Gamma; H^1(\cD))}  \nonumber \\ 
  & \leq   C\,\left(h\,\|y-y_h^\gamma(u)\|_{L^2(\Gamma; L^2(\cD))}
    + \sum_{n=1}^N (q_n+2)\,\left(\frac{\gamma_n\,r_n}{2}\right)^{q_n+1}\right).
\end{align}

Finally, combining \eqref{pr9}, \eqref{pr10}, and \eqref{eq:FE_error_h1_splitting}, and using Young's inequality yield
\begin{equation*}
  \|y-y_h^\gamma(u)\|_{L^2(\Gamma; L^2(\cD))}^2 \le C\left(h^2 
  + \sum_{n=1}^N(q_n+2)\,\left(\frac{\gamma_n\,r_n}{2}\right)^{q_n+1}\right).
\end{equation*}
The assertion now follows after extracting the root.
\end{proof}

\begin{lemma}
\label{lem:estimate_normal_derivatives}
Let $p\in L^2(\Gamma; H^1_0(\cD) \cap H^2(\cD))$ be the adjoint state from \eqref{eq:opt_pb} and let $w_h\in U_h$ be arbitrary. Then, the following error estimate holds
\begin{align*}
	a[S_h^\gamma (w_h), p]_\rho
	\le C\,\left(h^{1/2} +  h^{-1/2}\sum \limits_{n=1}^N (q_n+2)\left(\frac{\gamma_n\,r_n}{2}\right)^{q_n+1}\right)\|w_h\|_{L^2(\partial\cD)}.
\end{align*}
\end{lemma}
\begin{proof}
  The definition of $S_h^\gamma$, the boundedness of the bilinear form $a[\cdot, \cdot]_{\rho}$, the inequality \eqref{ritz_ineq}, and the estimate \eqref{eq:error_est_dual_solution} give
\begin{align}\label{eq:adjoint_error_proof_1}
a[S_h^\gamma (w_h), p]_\rho  & = a[S_h^\gamma (w_h), p - R_h(p)  + R_h((I - \Pi_{\gamma})p)]_\rho\nonumber \\
&\le C\,\|S_h^\gamma w_h\|_{L^2(\Gamma; H^1(\cD))}\,
\left(h\,\|p\|_{L^2(\Gamma; H^2(\cD))} + \sum_{n=1}^N \left(\frac{\gamma_n}{2}\right)^{q_n+1}\,\frac{\|\partial_{\xi_n}^{q_n+1}p\|_{L^2(\Gamma;H^1(\cD))}}{(q_n+1)!} \right).
\end{align}
To bound the norm involving $S_h^\gamma (w_h)$ we apply the estimate \eqref{ineq:yh_Piu_inv} (note that $b_N=0$ and $f_N=0$) with the stability of $\Pi_h$ which yields
\begin{equation}
  \label{eq:bound_Shgamma_by_discrete_bd_data}
  \|S_h^\gamma (w_h)\|_{H^1(\cD)} \le C\,h^{-1/2}\,\|w_h\|_{L^2(\partial\cD)}.
\end{equation}
By a standard shift theorem we obtain
\begin{equation}
  \label{eq:shift_theorem_adjoint_state}
  \|p\|_{L^2(\Gamma;H^2(\cD))}\le C\,\|y-y^d\|_{L^2(\Gamma; L^2(\cD))}\le C.
\end{equation}
To bound the term involving the derivatives $\partial_{\xi_n}^{q_n+1} p$ we use Lemma~\ref{lem:regularity_derivatives_xi} to get
\begin{equation*}
  \frac{\|\partial_{\xi_n}^{q_n+1}p\|_{L^2(\Gamma;H^1(\cD))}}{(q_n+1)!}
  \le
  C\,\sum_{j=0}^{q_n+1} \frac1{j!}\,r_n^{q_n+1-j}\,\|\partial_{\xi_n}^j(y-y^d)\|_{L^2(\Gamma; L^2(\cD))}.
\end{equation*}
The desired state $y^d$ is deterministic, so $\partial_{\xi_n} y^d \equiv 0$ for all $n=1,\ldots,N$. Furthermore, we use Corollary~\ref{cor:state_regularity_wrt_stochastic} to deduce
\begin{equation*}
  \frac1{j!}\,\|\partial_{\xi_n}^j(y-y^d)\|_{L^2(\Gamma; L^2(\cD))}
  \le C\,r_n^{j}
\end{equation*}
and thus,
\begin{equation}\label{eq:adjoint_state_stochastic_derivaties}
  \frac{\|\partial_{\xi_n}^{q_n+1}p\|_{L^2(\Gamma;H^1(\cD))}}{(q_n+1)!}
  \le C\,\sum_{j=0}^{q_n+1} r_n^{q_n+1} = C\,(q_n+2)\,r_n^{q_n+1}.
\end{equation}

Insertion of \eqref{eq:bound_Shgamma_by_discrete_bd_data}, \eqref{eq:shift_theorem_adjoint_state} and \eqref{eq:adjoint_state_stochastic_derivaties} into \eqref{eq:adjoint_error_proof_1} then yields
\begin{equation*}
  a[S_h^\gamma (w_h), p]_\rho \le C\,h^{-1/2}\,\|w_h\|_{L^2(\partial\cD)} \left(h
    + \sum_{n=1}^N (q_n+2)\,\left(\frac{\gamma_n\,r_n}{2}\right)^{q_n+1}\right).
\end{equation*}
\end{proof}

We are now in the position to state the main result of this section:
\begin{theorem}
\label{thm:final_estimate}
Suppose $y^d\in H^t(\cD)$ for some $t>0$.
Let $(y,u,p)\in L^2(\Gamma;H^1(\cD))\times U^{\textup{ad}} \times L^2(\Gamma;H^2(\cD)\cap H^1_0(\cD))$  be the solution triplet of \eqref{eq:opt_p} and denote by $(y_h^\gamma,u_h,p_h^\gamma)\in \mathcal Y_h^\gamma\times U_h^{\textup{ad}} \times \mathcal V_h^\gamma$ their approximation from \eqref{eq:opt_ph}. The following error estimate is valid:
\begin{align*}
  \|u-u_h\|_{L^2(\partial\cD)} + \|y-y_h^\gamma\|_{L^2(\Gamma;L^2(\cD))}\le C\left(h^{1/2}
  +  h^{-1/2}\sum \limits_{n=1}^N \left(\frac{\gamma_n\,r_n}{2}\right)^{q_n+1}
  \right).
\end{align*}
\end{theorem}
\begin{proof}
	The estimate follows after insertion of Lemmata~\ref{lem:projection_estimate}--\ref{lem:estimate_normal_derivatives} into Theorem~\ref{thm:general_estimate}. In addition, apply Young's inequality to the right-hand side of the estimate from Lemma~\ref{lem:FE_error_state_eq}.
\end{proof}

The result of Theorem~\ref{thm:final_estimate} is a worst-case estimate for the case that the $\cD$ is a convex polygonal domain. By exploiting the knowledge on the singularities at corners of the domain one can even show better convergence results. As this would exceed the scope of this paper we only comment on possible improvements:
\begin{remark}
  \begin{enumerate}[label=\alph*)]
  \item A more precise version of Lemma~\ref{lem:projection_estimate} can be found in \cite[Lemma 5.6]{TApel_MMateos_JPfefferer_ARosch_2018}. If $\cD$ is a convex polygonal domain, with $\omega$ denoting the largest opening angle of the corners, there holds the estimate
    \begin{equation*}
      \|u-u_h^*\|_{L^2(\partial\cD)} \le C\,h^s,\quad s < \min\{3/2,\pi/\omega-1/2\}.
    \end{equation*}
    To establish this result, the assumption $y^d\in H^t(\cD)$ for all $t<\min\{1,\pi/\omega-1\}$ is necessary. In case of a smooth domain and sufficiently regular data, one even achieves a convergence rate of almost $3/2$.
  \item The result of Lemma~\ref{lem:FE_error_state_eq} can also be improved taking into account higher regularity of the optimal state. Following known results for the deterministic setting from \cite{TApel_MMateos_JPfefferer_ARosch_2015}, one can show $u\in H^{s}(\partial\cD)$ and $y\in L^2(\Gamma;H^{s+1/2}(\cD))$ with $s<\min\{3/2,\pi/\omega-1/2\}$ in case of convex and polygonal $\cD$. This allows to apply an improved error estimate for the $L^2$-projection of $u+b_N$ in \eqref{pr9}, leading to $E_2\le C\,h^{s+1/2}$, and an improved estimate for $\|y-y_h^\gamma(u)\|_{L^2(\Gamma;H^1(\cD))}$ in \eqref{pr10} leading to $E_1\le C\,h^{s+1/2} + c(q)$.
  \item Further improvements in Lemma~\ref{lem:estimate_normal_derivatives} are also possible, but require very advanced techniques in the proofs. In the simplest case, assume that $p\in L^2(\Gamma;W^{2,\infty}(\cD))$ holds, which would follow, e.g., under the quite restrictive assumption $\omega < \pi/2$. In this case, instead of \eqref{eq:adjoint_error_proof_1}, we can estimate as follows:
    \begin{align*}
      a[S_h^\gamma(w_h),p]_\rho
      &= a[\widetilde S_h^\gamma(w_h), p-p_h^\gamma]
      \\
      &\le C\,\|\widetilde S_h^\gamma(w_h)\|_{W^{1,1}(\cD)}\,\|p-p_h^\gamma\|_{L^2(\Gamma;W^{1,\infty}(\cD))}
    \end{align*}
    with $p_h^\gamma\in \mathcal V_h^\gamma$ satisfying $a[v_h^\gamma,p-p_h^\gamma]=0$ for all $v_h^\gamma\in \mathcal V_h^\gamma$ and $\widetilde S_h^\gamma$ the zero extension, mapping $w_h\in U_h$ to that function from $\mathcal Y_h^\gamma$ vanishing in the interior nodes. It remains to use the estimate $\|\widetilde S_h^\gamma(w_h)\|_{W^{1,1}(\cD)}\le C\,\|w_h\|_{L^1(\partial\cD)} \le C\,\|w_h\|_{L^2(\partial\cD)}$ (see \cite[Lemma 3.3]{SMay_RRannacher_BVexler_2013}) and the $W^{1,\infty}(\cD)$-error estimate $\|p-R_hp\|_{W^{1,\infty}(\cD)} \le C\,h\,|p|_{W^{2,\infty}(\cD)}$ (see \cite{RRannacher_RScott_1982}) to conclude
    \begin{equation*}
      a[S_h^\gamma (w_h), p]_\rho
      \le C\,\left(h +  \sum \limits_{n=1}^N \left(\frac{\gamma_n\,r_n}{2}\right)^{q_n+1}\right)\|w_h\|_{L^2(\partial\cD)}.
    \end{equation*}
    In the case $\omega \in (\pi/2,\pi)$ we refer to \cite{TApel_MMateos_JPfefferer_ARosch_2018,JPfefferer_MWinkler_2019} where sharp discretization error estimates with respect to $h$ are proved in the deterministic setting. An extension to the stochastic setting is not straight-forward and might be subject of future research.
  \end{enumerate}
\end{remark}

\section{Numerical solution}\label{sec:linear_system} 

In this section, we first provide the matrix formulation of the discretized optimality system by employing a ``discretize-then-optimize'' approach; see, e.g., \cite{FTroeltzsch_2010a}, for both unconstrained and constrained problems. Then, we introduce suitable preconditioners to solve the corresponding linear systems.

\subsection{Unconstrained problem}

By following the Cameron--Martin theorem \cite{RHCameron_WTMartin_1947a}, state $y_h^\gamma(\boldsymbol{x},\xi)$ and  adjoint state $p_h^\gamma(\boldsymbol{x},\xi)$ on the finite parametric domain $\Gamma \subset \mathbb{R}^N$ 
are characterized by a finite generalized polynomial chaos (gPC) approximation
\begin{equation}
		y_h^\gamma(\boldsymbol{x},\xi) = \sum \limits_{i=1}^{N_{\xi}} y_{i,h}(\boldsymbol{x})\,\psi_i(\xi), \qquad 
		p_h^\gamma(\boldsymbol{x},\xi) = \sum \limits_{i=1}^{N_{\xi}} p_{i,h}(\boldsymbol{x})\,\psi_i(\xi),
\end{equation}
where the finite element functions $y_{i,h}\in Y_h$ and $p_{i,h}(\boldsymbol{x})\in V_h$ are the deterministic modes of the expansion and $N_{\xi}$ is the number of total generalized polynomial chaos  basis, given by
\begin{equation*}	
	N_{\xi} = 1 + \sum \limits_{s=1}^{Q} \frac{1}{s!} \prod \limits_{j=0}^{s-1} (N + j) = \frac{(N+Q)!}{N!Q!},
\end{equation*}
with $ N $ the dimension of the random vector $\xi$   and  $ Q $  the highest order in the stochastic (parametric) basis set.

The coefficients of the finite element functions with respect to the standard basis are $\boldsymbol{y}_i^{0} = (y^{0}_{i,1},\ldots,y^{0}_{i,\N_{\textup{I}}})^\top$ for $i=1,\ldots,N_{\xi}$, and for problems with inhomogeneous Dirichlet data $\boldsymbol{u} = (u_1,\ldots,u_{\N_\textup{B}})^\top$. The vectors of unknowns used to represent the interior and boundary values of the state $y_h^\gamma$ are then $\boldsymbol y^0 = \text{vec}[\boldsymbol{y}_1^{0},\ldots,\boldsymbol{y}^0_{N_{\xi}}]$ and  $\boldsymbol{u}$, respectively.
As basis for the stochastic space $\mathcal S^\gamma$ we choose the orthonormal polynomials $\{\psi_i\}_{i=1}^{N_\xi}$ satisfying 
\begin{equation*}
  \int_\Gamma \psi_i(\xi)\,\psi_j(\xi)\,\rho(\xi)\,\d\xi=\delta_{i,j},\ i,j=1,\ldots,N_{\xi},
\end{equation*}
that can be constructed from univariate orthonormal polynomials, see \cite{OGErnst_EUllmann_2010}. By construction, the first basis functions are
\begin{equation*}
	\psi_1(\xi) = 1,\quad \psi_{j+1}(\xi) = \xi_j,\ j=0,\ldots,N_{\xi}.
\end{equation*}
Recall the KL expansions from \eqref{eqn:kltrun} for $a(\boldsymbol{x},\omega)$, $b(\boldsymbol{x},\omega)$, and $f(\boldsymbol{x},\omega)$, with the notation $a_k = \kappa_a \, \sqrt{\lambda_k}\,w_k(\boldsymbol{x})$, $b_k = \kappa_b \, \sqrt{\lambda_k}\,w_k(\boldsymbol{x})$, and $f_k = \kappa_f \, \sqrt{\lambda_k}\,w_k(\boldsymbol{x})$, respectively, for $k=1, \dots, N$. Then, we define a fully-discrete approximation of $b$ by taking its KL expansion and by projecting the coefficient functions $b_j(\boldsymbol x)$ into  the space $U_h$. This gives
\begin{align*}
	b_{N,h} := \Pi_h (b_N) = \sum_{j=1}^N \Pi_h(b_{j})\,\xi_j 
	 = \sum_{j=1}^N \sum_{n=1}^{\mathcal N_B} b_{j,n}\,\phi_{\mathcal N_I+n}\,\xi_j 
\end{align*}
with coefficients $b_{j,n}\in\R$ for $j=1,\ldots,N$, $n=1,\ldots,\N_{\textup{B}}$.
We summarize these coefficients into vectors $[\boldsymbol b_k]_n = b_{k,n}$, $n=1,\ldots,\N_B$, and define furthermore $\boldsymbol b = \text{vec}[\boldsymbol b_1,\ldots, \boldsymbol b_N]$.

The coefficients introduced above allow the following representation of the state variable
\begin{align}\label{eq:ansatz_discrete_state}
  y_h^\gamma(\bx,\xi)
	&=
  \sum_{n=1}^{\N_{\text{I}}}\sum_{j=1}^{N_{\xi}}  y_{j,n}^0\,\psi_j(\xi)\,\phi_n(\boldsymbol{x})
  + \sum_{n=1}^{\N_{\text{B}}} \left(u_n\, 
  + \sum_{j=1}^N  b_{j,n}\,\xi_j\right)\phi_{\N_{\textup{I}}+n}(\bx).
\end{align}
For each test function $v=\psi_i\,\phi_m$,  $m=1,\ldots,\N_{\textup{I}}$, $i=1,\ldots,N_{\xi}$, we then obtain from \eqref{eqn:vari_disc} the system of linear equations
\begin{align}
  \label{eq:state_eq_ansatz}
 &  \sum_{j=1}^{N_{\xi}}\sum_{n=1}^{\N_{\textup{I}}}\Bigg(\left(\int_\Gamma\psi_j\,\psi_i\,\rho\,\d \xi\right)\left(\int_\cD \overline a\,\nabla\phi_n\cdot\nabla \phi_m\,\d \boldsymbol{x}\right) \nonumber\\
 & \qquad + \sum_{k=1}^N \left(\int_\Gamma\xi_k\,\psi_j\,\psi_i\,\rho\,\d \xi\right)\left(\int_\cD a_k\,\nabla\phi_n\cdot\nabla\phi_m\,\d \boldsymbol{x}\right)\Bigg) y_{j,n}^{0}\nonumber\\
 & +   \sum_{n=1}^{\N_{\textup{B}}}\Bigg(\left(\int_\Gamma\psi_i\,\rho\,\d \xi\right)\left(\int_\cD \overline a\,\nabla\phi_{\N_{\textup{I}+n}}\cdot\nabla \phi_m\,\d \boldsymbol{x}\right) \nonumber\\
 & \qquad  + \sum_{k=1}^N \left(\int_\Gamma\xi_k\,\psi_i\,\rho\,\d \xi\right)\left(\int_\cD a_k\,\nabla\phi_{\N_{\textup{I}+n}} \cdot\nabla\phi_m\,\d \boldsymbol{x}\right)\Bigg) u_n \nonumber\\
 & =\left(\int_\Gamma \psi_i\,\rho\,\d \xi\right) \left(\int_{\cD} \overline f\,\phi_m\,\d \boldsymbol{x}\right)+
     \sum_{k=1}^N \left(\int_\Gamma \xi_k\,\psi_i\,\rho\,\d \xi\right) \left(\int_{\cD} f_k\,\phi_m\,\d \boldsymbol{x}\right) \nonumber\\
 & \qquad
 - \sum_{l=1}^N\sum_{n=1}^{\mathcal{N}_B} \Bigg( \left(\int_\Gamma \xi_l\,\psi_i\,\rho\,\d \xi\right) \left(\int_{\cD} \bar a\,\nabla\phi_{\N_{\textup{I}} +n}\cdot\nabla\phi_m \,\d \boldsymbol{x}\right) \nonumber \\
 &\qquad\qquad + \sum_{k=1}^N \left(\int_{\Gamma} \xi_k\,\xi_l\,\psi_i\,\rho\d \xi\right) \left(\int_{\cD} a_k\,\nabla\phi_{\N_{\textup{I}} +n}\cdot\nabla\phi_m\,\d \bx\right) \Bigg)b_{l,n}.
\end{align}
The integrals in \eqref{eq:state_eq_ansatz} motivate the definition of the matrices
\begin{align*}  
  [\overline P]_{i,j} &= \int_\Gamma\psi_j\,\psi_i\,\rho \, \d \xi,
  & [P_k]_{i,j} &= \int_\Gamma\xi_k\,\psi_j\,\psi_i\,\rho\,\d \xi,   &   i,j&=1,\ldots,N_{\xi},\, k=1,\ldots,N, \\
  [\overline K]_{m,n} &= \int_{\cD} \overline a\,\nabla\phi_n\cdot\nabla\phi_m\,\d \boldsymbol{x},
  &[K_k]_{m,n} &= \int_{\cD} a_k\,\nabla \phi_n\cdot \nabla \phi_m\,\d \boldsymbol{x}, &m,n&=1,\ldots,\N,\, k=1,\ldots,N, \\
  [M]_{m,n} &= \int_{\cD} \phi_n\, \phi_m\,\d \boldsymbol{x}, & [M^\partial]_{r,s} &= \int_{\partial\cD} \phi_s\,\phi_r\,\d s_{\boldsymbol{x}}, &m,n&=1,\ldots,\N, \, r,s=1, \ldots, \N_{\textup{B}},\\
\intertext{and vectors}
  [\overline p]_i &= \int_\Gamma\psi_i\,\rho\,\d \xi,
  & [p_k]_i &= \int_\Gamma\xi_k\,\psi_i\,\rho\,\d \xi, &   i&=1,\ldots,N_{\xi},\, k=1,\ldots,N, \\
  [\underline{\overline{f}}]_m &= \int_\cD \overline f\,\phi_m \,\d \boldsymbol{x},& 
  [\underline{f}_k]_m &= \int_\cD f_k\,\phi_m\,\d \boldsymbol{x},  &m&=1,\ldots,\N,\, k=1,\ldots,N,  \\
  [\underline{y}^d]_m &= \int_\cD y^d\, \phi_m\,\d \boldsymbol{x},&&  &m&=1,\ldots,\N. 
\end{align*}
The choice of basis functions, i.e., Legendre polynomials,  implies $\overline P = I$, $\overline p = e_1$, and $p_k = e_{k+1}$. Furthermore, for some matrix $P$ we denote by $P_{R \times C}$, $R,C\in \mathbb N$, the submatrices containing the first $R$ rows and first $C$ columns of $P$.

Due to the numbering of the degrees of freedom, we get a special structure for the finite element matrices and vectors according to 
\begin{equation*}
  K = \begin{bmatrix}
    K_{\textup{II}} & K_{\textup{IB}}
    \\
    K_{\textup{BI}} & K_{\textup{BB}}
  \end{bmatrix},
  \qquad
  \underline f = \begin{bmatrix}
    \underline f_{\textup{I}} \\
   \underline f_{\textup{B}}
  \end{bmatrix}.
\end{equation*}
Here, the first indices $\textup{I}$ and $\textup{B}$ indicate that the integration is performed with respect to the test functions related to the inner vertices and the boundary vertices, respectively, whereas the second index means that the function to be integrated is a linear combination of basis functions related to the inner and boundary vertices, respectively.  

We may now introduce the stiffness and mass matrices obtained by the stochastic Galerkin method 
\begin{subequations}
    \label{eq:kronecker_matrices}
\begin{align}
  \mathcal K_{\textup{II}} &= \overline P\otimes \overline K_{\textup{II}} + \sum_{k=1}^N P_k \otimes K_{k,\textup{II}},
\label{eq:kronecker_matrices_II}\\
  \cK_{\textup{IB}} &= \overline P_{N_{\xi}\times N+1} \otimes \overline K_{\textup{IB}}
                      + \sum_{k=1}^N P_{k,N_{\xi}\times N+1} \otimes K_{k,\textup{IB}} = \mathcal{K}_{\textup{BI}}^\top,\\  
  \cM_{\textup{II}} &= \overline P\otimes M_{\textup{II}}, \\
  \cM_{\textup{IB}} &= \overline P_{N_{\xi}\times N+1}\otimes M_{\text{IB}}=\cM_{\textup{BI}}^\top,
\end{align}
\end{subequations}
and  the assembled vectors,
\begin{subequations}
\begin{align}
  \label{eq:rhs_assembled}
  \boldsymbol{f}_{\textup{I}} &= \overline p \otimes \underline{\overline{f}}_\textup{I} + \sum_{k=1}^N p_k\otimes \underline{f}_{k,\textup{I}},
  \\
  \boldsymbol{y}^d_{\textup{I}} &= \overline p\otimes \underline{y}^d_\textup{I} ,\quad \boldsymbol{y}^d_{\textup{B}} = \underline y^d_{\textup{B}}.
\end{align}
\end{subequations}
We furthermore split up the matrix $\mathcal K_{\textup{IB}}$ into the columns related to $u$ and $b_N$, respectively, i.\,e.,
\begin{equation*}
  \mathcal K_{\textup{IB}} = \begin{bmatrix} \mathcal K_{\textup{IB}}^u & \mathcal K_{\textup{IB}}^b\end{bmatrix}, 
\end{equation*}
and the submatrices have dimension $\cK_{\textup{IB}}^u \in \R^{N_{\xi}\cdot\N_{\textup{I}}\times \N_{\textup{B}}}$, $\cK_{\textup{IB}}^b\in \R^{N_{\xi}\cdot\N_{\textup{I}}\times N\cdot\N_{\textup{B}}}$. Analogously, we may split the partial mass matrix into $\mathcal M_{\textup{IB}} = [\mathcal M_{\textup{IB}}^u\ \mathcal M_{\textup{IB}}^b]$.

With this notation at hand the discrete state equation \eqref{eq:state_eq_ansatz} can be equivalently written in matrix-vector notation as
\begin{equation*}
  \mathcal K_{\textup{II}}\,\boldsymbol{y}^0 + \cK_{\textup{IB}}^u\boldsymbol u 
  =
  \boldsymbol{f}_{\textup{I}} - \cK_{\textup{IB}}^b \boldsymbol b
\end{equation*}
and in the same way, the objective functional \eqref{eqn:objec_disc} becomes, when neglecting the constant terms independent of $\boldsymbol y^0$ and $\boldsymbol u$,
\begin{equation} \label{eq:dicrete_J}
  \mathcal{J}(\boldsymbol{y}^0,\boldsymbol{u})
  =
  \frac12\,\begin{bmatrix}\boldsymbol{y}^0\\ \boldsymbol{u} \\ \boldsymbol{b}\end{bmatrix}^\top
  \left[
  \begin{array}{ccc}
    \cM_{\textup{II}} & \cM_{\textup{IB}}^u & \cM_{\textup{IB}}^b \\
    \cM_{\textup{BI}}^u & M_{\textup{BB}} & 0 \\
	\cM_{\textup{BI}}^b & 0 & 0
  \end{array}
  \right]
  \begin{bmatrix}
    \boldsymbol{y}^0\\ \boldsymbol{u} \\ \boldsymbol{b}
  \end{bmatrix}
  -
  \begin{bmatrix}
    \boldsymbol{y}^0\\ \boldsymbol{u}
  \end{bmatrix}^\top
  \begin{bmatrix}
    \boldsymbol{y}_{\textup{I}}^d \\
    \boldsymbol{y}_{\textup{B}}^d
  \end{bmatrix}
  + \frac\alpha2\,\boldsymbol{u}^\top M^\partial\boldsymbol{u}.
\end{equation}
It is easy to check that differentiating the Lagrangian
$\mathcal L\colon \R^{N_{\xi}\cdot \N_{\textup{I}}} \times \R^{\N_{\textup{B}}} \times \R^{N_{\xi}\cdot \N_{\textup{I}}}$,
\begin{align*}
  \mathcal L(\boldsymbol{y}^0,\boldsymbol{u}, \boldsymbol{p})
  &=
\mathcal{J}(\boldsymbol{y}^0,\boldsymbol{u}) + \boldsymbol{p}^\top\left(\mathcal K_{\textup{II}}\,\boldsymbol{y}^0 + \cK_{\textup{IB}}^u\boldsymbol u + \cK_{\textup{IB}}^b \boldsymbol b - \boldsymbol{f}_{\textup{I}}\right)
\end{align*}
with respect to the state vector $\boldsymbol{y}^0$ gives the adjoint equation \eqref{eq:adjoint_equation_discrete}
\begin{equation*}
	\cK_{\textup{II}}^\top \,\boldsymbol{p} + \cM_{\textup{II}} \boldsymbol{y}^0
	+ \cM_{\textup{IB}}^u\, \boldsymbol{u} = \boldsymbol{y}^d_{\textup{I}} - \cM_{\textup{IB}}^b\boldsymbol b,
\end{equation*}
and with respect to the control vector $\boldsymbol{u}$ gives the optimality condition \eqref{eq:opt_cond_discrete}
\begin{equation*}
	(\alpha\,M^\partial + M_{\textup{BB}}) \,\boldsymbol u + \mathcal K_{\textup{BI}}^u\,\boldsymbol{p} + \mathcal M_{\textup{BI}}^u\,\boldsymbol{y}^0
	 = \boldsymbol{y}^d_{\textup{B}}.
\end{equation*}
To conclude the previous investigations, we get the following system of linear equations, equivalent to \eqref{eq:opt_ph},
\begin{equation}\label{eq:opt_cond_matrix_vector}
  \begin{bmatrix}
    \cM_{\textup{II}} & \cM_{\textup{IB}}^u & \cK_{\textup{II}}^\top \\
    \cM_{\textup{BI}}^u & M_{\textup{BB}} + \alpha M^\partial & \cK_{\textup{BI}}^u \\
    \cK_{\textup{II}} & \cK_{\textup{IB}}^u & 0
  \end{bmatrix}
  \begin{bmatrix}
    \boldsymbol{y}^0 \\
    \boldsymbol{u} \\
    \boldsymbol{p}
  \end{bmatrix}
  =
  \begin{bmatrix}
    \boldsymbol{y}^d_{\textup{I}} - \cM_{\textup{IB}}^b \boldsymbol b \\
    \boldsymbol{y}^d_{\textup{B}} \phantom{- \cM_{\textup{IB}}^b \boldsymbol b}\\
    \boldsymbol{f}_{\textup{I}} - \mathcal K_{\textup{IB}}^b\boldsymbol{b}
  \end{bmatrix}.
\end{equation}

\subsection{Preconditioning of the system matrix}
\label{sec:preconditioning}

We want to solve \eqref{eq:opt_cond_matrix_vector}  with an iterative solver like \textsc{minres}. Without tailored preconditioners the iteration numbers become incredibly large, in particular when the regularization parameter $\alpha$ and the discretization parameters become small. The aim is thus to construct a preconditioner which is robust with respect to these parameters. We utilize the general approach from \cite{PBenner_AOnwunto_MStoll_2016}, where a preconditioner for a stochastic optimal control problem with distributed control is derived and extend the preconditioner derived for a deterministic control problem with Dirichlet boundary control from \cite{StollWinkler2021} to the case that the problem contains uncertain parameters.

The ideal preconditioner for \eqref{eq:opt_cond_matrix_vector} would be the block-diagonal matrix
\begin{equation*}
  \mathcal{P} = \begin{bmatrix}
        \cM_{\textup{II}} & \cM_{\textup{IB}}^u & 0 \\
    \cM_{\textup{BI}}^u & M_{\textup{BB}} + \alpha M^\partial & 0 \\
    0 & 0 & \mathcal S
  \end{bmatrix}
\end{equation*}
with the Schur complement
\begin{equation*}
  \mathcal S = \begin{bmatrix} \cK_{\text{II}} &  \cK_{\text{IB}}^u\end{bmatrix} \begin{bmatrix}\cM_{\text{II}} & \cM_{\text{IB}}^u \\ \cM_{\text{BI}}^u & M_{\text{BB}} + \alpha\,M^\partial\end{bmatrix}^{-1} \begin{bmatrix} \cK_{\text{II}} \\ \cK_{\text{BI}}^u\end{bmatrix}.
\end{equation*}
The $2\times 2$ block involving the mass matrices can be also approximated by their Schur complement
\begin{equation*}
  \begin{bmatrix}\cM_{\text{II}} & \cM_{\text{IB}}^u \\ \cM_{\text{BI}}^u & M_{\text{BB}} + \alpha\,M^\partial\end{bmatrix}
  \approx \widehat{\mathcal S}
  :=
  \begin{bmatrix}
    \cM_{\text{II}} & 0 \\ 0 & \mathcal S_{\cM}
  \end{bmatrix}
\end{equation*}
with
\begin{equation*}
  \mathcal S_{\cM} = (M_{\text{BB}} + \alpha M^\partial) - \cM_{\text{BI}}^u \,\cM_{\text{II}}^{-1} \cM_{\text{IB}}^u.
\end{equation*}
Combining the above ideas yields the block-diagonal preconditioner
\begin{equation*}
  \mathcal P\approx\begin{bmatrix}
    \cM_{\text{II}} & & \\ & \mathcal S_{\cM} & \\ & & \mathcal S
  \end{bmatrix}.
\end{equation*}
It remains to approximate the 3 diagonal blocks appropriately, so that this preconditioner is cheap to apply but maintaining the robustness. For the first block, containing $\cM_{\text{II}}$ we recall the structure $\cM_{\text{II}} = \overline P \otimes M_{\text{II}}$, see \eqref{eq:kronecker_matrices_II}, and as $\overline P$ is a diagonal matrix, $\cM_{\text{II}}$ is also block-diagonal. Applying the inverse of $\cM_{\text{II}}$ thus requires solving $N_{\xi}$ systems of linear equations with system matrix $M_{\text{II}}$. For the second block, involving $\mathcal S_{\cM}$ we would have to invert a sum of two matrices which is not very practical. However, numerical experiments show that omitting the latter term yields very promising results. To this end, the approximation that is used is
\begin{equation*}
  \mathcal S_{\cM} \approx \widehat{\mathcal{S}}_{\cM} := M_{\text{BB}} + \alpha\,M^\partial,
\end{equation*}
which is of small dimension $\N_{\textup{B}}\times \N_{\textup{B}}$ and can be explicitly assembled. For the Schur complement of the overall system
\begin{equation*}
	\mathcal S= \cK_{\text{II}}\,\cM_{II}^{-1}\,\cK_{\text{II}} + \cK_{\text{IB}}^u \,\mathcal S_{\cM}^{-1}\,\cK_{\text{BI}}^u
\end{equation*}
 we also omit the second term and arrive at
\begin{align*}
  \widehat{\mathcal S} &= \widehat{\cK}_{\text{II}}\,\cM_{\text{II}}^{-1}\,\widehat{\cK}_{\text{II}}
\end{align*}
with suitable approximations $\widehat \cK_{\textup{II}}$ for the stiffness matrix $\cK_{\textup{II}}$. Our experiments have shown that in case of $Q\le 3$ the mean based preconditioner
\begin{equation*}
	\widehat \cK_{\textup{II}} := \overline P\otimes \overline K_{\textup{II}}
\end{equation*}
works very well so that the resulting preconditioner is robust with respect to the spatial discretization. As $\overline P$ is a diagonal matrix, the application of $\widehat \cK_{\textup{II}}^{-1}$ requires to solve $N_\xi$ linear systems with system matrix $K_{\textup{II}}$.
For $Q\ge 4$, however, the iteration numbers for \textsc{minres} drastically increase. In this case the preconditioner proposed by Ullmann \cite{EUllmann_2010}, i.\,e., choosing
\begin{equation*}
	\widehat \cK_{\textup{II}} := (G\otimes I_{\mathcal{N}_{\textup{I}}})(\overline P\otimes \overline K_{\textup{II}})
\end{equation*}
with
\begin{equation*}
	G = \overline P + \sum_{k=1}^N \frac{\text{tr}(K_{k,\textup{II}}^\top\,\overline K_{\textup{II}})}{\text{tr}(\overline K_{\textup{II}}^\top\,\overline K_{\textup{II}})} P_k
\end{equation*}
yields much better results.
Note that the application of $\widehat\cK_{\textup{II}}^{-1}$ requires in this case additionally $\mathcal{N}_{\textup{I}}$ solves of linear systems with system matrix $G$.

\subsection{Constrained problem}
\label{sec:constrained_problem}

In this section we briefly discuss the extension to the case that box constraints on the control variable are present. The discretized optimal control problem, analogous to \eqref{eq:dicrete_J},  reads as follows 
\begin{align}\label{eq:dc1}
 \min \; \mathcal{J}(\boldsymbol{y}^0,\boldsymbol{u}), \quad \hbox{subject to} \quad \boldsymbol{u}_a \leq \boldsymbol{u} \leq \boldsymbol{u}_b,
\end{align}
where $\boldsymbol{u}_a,\boldsymbol{u}_b \in \mathbb{R}^{\mathcal{N}_B}$ are the vectors $\boldsymbol{u}_a = u_a\,\vec 1, \boldsymbol{u}_b=u_b\,\vec 1$, and the inequalities are understood component-wise. Then, if $\boldsymbol{u}$ is the solution of \eqref{eq:dc1}, one concludes the existence of unique  Lagrange multipliers $\boldsymbol{\lambda}_a, \boldsymbol{\lambda}_b \in \mathbb{R}^{\mathcal{N}_B}$  such that 
\begin{subequations}\label{eq:dc2}
\begin{align}
(\alpha\,M^\partial+M_{\textup{BB}}) \,\boldsymbol u + \mathcal K_{\textup{BI}}^u\,\boldsymbol{p} + \mathcal M_{\textup{BI}}^u\,\boldsymbol{y}^0
	- \boldsymbol{y}^d_{\textup{B}} + \underbrace{\boldsymbol{\lambda}_b - \boldsymbol{\lambda}_a}_{\boldsymbol{\lambda}} &= 0, \\
\boldsymbol{\lambda}_a \geq 0, \; \boldsymbol{u}_a \leq \boldsymbol{u}, \; \boldsymbol{\lambda}_a^T (\boldsymbol{u}_a - \boldsymbol{u})  &=  0,\label{eq:compl_cond_1} \\
\boldsymbol{\lambda}_b \geq 0, \; \boldsymbol{u} \leq \boldsymbol{u}_b, \; \boldsymbol{\lambda}_b^T (\boldsymbol{u} - \boldsymbol{u}_b) &=  0.\label{eq:compl_cond_2}
\end{align}
\end{subequations}
The complementary conditions \eqref{eq:compl_cond_1},\eqref{eq:compl_cond_2} can also be expressed by
\begin{align}\label{eq:dc3}
\boldsymbol{\lambda} - \max\{\boldsymbol{0}, \boldsymbol{\lambda} + \sigma \, (\boldsymbol{u} - \boldsymbol{u}_b) \}  
                     - \min\{\boldsymbol{0}, \boldsymbol{\lambda} - \sigma \, (\boldsymbol{u}_a - \boldsymbol{u})\} = 0,
\end{align}
where $\sigma > 0 $ is an arbitrary fixed real number and the $\max$- and $\min$-operations are understood component-wise. For every pair $(\boldsymbol{u}, \boldsymbol{\lambda})$, we define the sets of active and inactive indices by
\begin{subequations}\label{eq:dc4}
\begin{align}
\mathcal{A}_a(\boldsymbol{u}, \boldsymbol{\lambda}) &= \{ j\in \{1,\ldots,\mathcal N_{\textup{B}}\}: \, \lambda_j - \sigma \, (u_a - u_j) < 0 \}, \\
\mathcal{A}_b(\boldsymbol{u}, \boldsymbol{\lambda}) &= \{ j\in \{1,\ldots,\mathcal N_{\textup{B}}\}: \, \lambda_j + \sigma \, (u_j - u_b) > 0 \}, \\
\mathcal{I}(\boldsymbol{u}, \boldsymbol{\lambda}) &= \{1, \ldots, \mathcal{N}_{\textup B}\} \backslash \big( \mathcal{A}_a \cup \mathcal{A}_b \big).
\end{align}
\end{subequations}
Then, the complementarity conditions in \eqref{eq:compl_cond_1}--\eqref{eq:compl_cond_2} are equivalent to
\begin{subequations}\label{eq:dc5}
    \begin{align}
    \boldsymbol{u} &=\boldsymbol{u}_a, & \boldsymbol{\lambda}_b &=0, & \boldsymbol{\lambda} &\leq 0 && \hbox{on  }  \mathcal{A}_a(\boldsymbol{u}, \boldsymbol{\lambda}), \\
    \boldsymbol{u} &=\boldsymbol{u}_b,& \boldsymbol{\lambda}_a &=0, & \boldsymbol{\lambda} &\geq 0 && \hbox{on  }  \mathcal{A}_b(\boldsymbol{u}, \boldsymbol{\lambda}), \\
    \boldsymbol{u}_a &< \boldsymbol{u} < \boldsymbol{u}_b,& \boldsymbol{\lambda}_a&=\boldsymbol{\lambda}_b=0,& \boldsymbol{\lambda} &=0 &&\hbox{on  }  \mathcal{I}.
\end{align}
\end{subequations}

\begin{algorithm}[htp!]
\caption{Primal-dual active set (PDAS) strategy}
\label{Alg:pdas}
\begin{algorithmic}[lines]
    \STATE \textbf{Input}: Given the parameter $\sigma$.
    \STATE Set $k=0$.
    \STATE Set initial values for $\boldsymbol{u}^{(0)}$, $\boldsymbol{\lambda}_a^{(0)}$, and $\boldsymbol{\lambda}_b^{(0)}$.
    \STATE Compute indices of  the initial active sets  $\mathcal{A}_a^{(0)} := \mathcal{A}_a(\boldsymbol{u}^{(0)}, \boldsymbol{\lambda}^{(0)})$, $\mathcal{A}_b^{(0)}:=\mathcal{A}_b(\boldsymbol{u}^{(0)}, \boldsymbol{\lambda}^{(0)})$, 
          and the inactive set  $\mathcal{I}^{(0)}:=\mathcal{I}(\boldsymbol{u}^0, \boldsymbol{\lambda}^0)$ using \eqref{eq:dc4}.
    \FOR {$k=1,2, \ldots$}
        \STATE Solve (\ref{eq:dc6}) to obtain  $(\boldsymbol{u}^{(k)},\boldsymbol{\lambda}^{(k)})$.
        \STATE Compute indices of  active sets  $\mathcal{A}_a^{(k)} :=\mathcal{A}_a(\boldsymbol u^{(k)}, \boldsymbol{\lambda}^{(k)})$, $\mathcal{A}_b^{(k)}:= \mathcal{A}_b(\boldsymbol u^{(k)}, \boldsymbol{\lambda}^{(k)})$, and inactive set  $\mathcal{I}^{(k)}:=\mathcal{I}(\boldsymbol u^{(k)},\boldsymbol{\lambda}^{(k)})$ using \eqref{eq:dc4}.
        \IF{$\mathcal{A}_a^{(k)}=\mathcal{A}^{(k+1)}_a$, $\mathcal{A}^{(k)}_{b}=\mathcal{A}^{(k+1)}_{b}$, and $\mathcal{I}^{(k)}=\mathcal{I}^{(k+1)}$}
            \STATE STOP.
        \ENDIF
        \STATE Set $k:=k+1$.
    \ENDFOR
\end{algorithmic}
\end{algorithm}

Algorithm~\ref{Alg:pdas} summarizes the primal-dual active set (PDAS) strategy as a semi-smooth Newton step, see, e.g., \cite{MBergounioux_KIto_KKunisch_1999a,MMateos_2018}. The linear system to be solved in each iteration reads
\begin{equation}\label{eq:dc6}
  \begin{bmatrix}
    \cM_{\textup{II}} & \cM_{\textup{IB}}^u                       & \cK_{\textup{II}}^\top  &  \\
    \cM_{\textup{BI}}^u & M_{\textup{BB}} + \alpha\,M^\partial & \cK_{\textup{BI}}^u       & I \\
    \cK_{\textup{II}} & \cK_{\textup{IB}}^u                       &                         &                   \\
     & \sigma \big( I_{\mathcal{A}_a^{(k)}} + I_{\mathcal{A}_b^{(k)}} \big) &                   &I_{\mathcal{I}^{(k)}}
  \end{bmatrix}
  \begin{bmatrix}
    \boldsymbol{y}^0 \\
    \boldsymbol{u} \\
    \boldsymbol{p} \\
    \boldsymbol{\lambda}
  \end{bmatrix}
  =
  \begin{bmatrix}
    \boldsymbol{y}^d_{\textup{I}}- \cM_{\textup{IB}}^b \boldsymbol b \\
    \boldsymbol{y}^d_{\textup{B}} \\
    \boldsymbol{f}_{\textup{I}} - \cK_{\textup{IB}}^b \boldsymbol b\\
    \sigma \big( \chi_{\mathcal{A}_a^{(k)}} \boldsymbol{u}_a  + \chi_{\mathcal{A}_b^{(k)}} \boldsymbol{u}_b \big)
  \end{bmatrix},
\end{equation}
where $I$ is the identity matrix of dimension $\N_{\textup{B}}\times \N_{\textup{B}}$ and $I_{\mathcal S}$ the diagonal binary matrix with nonzero entries in the index set $\mathcal S\subset \{1,\ldots,\N_{\textup{B}}\}$.

The system \eqref{eq:dc6} can be made symmetric by eliminating $\boldsymbol{\lambda}|_{\mathcal I} = 0$ and by dividing the last block row by $\sigma$. This allows to apply tailored preconditioners for which we refer to \cite{JWPearson_JPestana_2020,MPorcelli_VSimoncini_MTani_2015}. We note that the value of $\sigma$ has no effect on the solution but it affects the update of the active  sets in \eqref{eq:dc4}. The choice $\sigma = \alpha$ is used in the numerical simulations.


\section{Numerical experiments}\label{sec:numeric} 

In this section, we carry out some numerical experiments to illustrate the validity and efficiency of our numerical approaches to solve the stochastic Dirichlet boundary control problem \eqref{eqn:objec}--\eqref{eqn:constPDE}. Throughout the following experiments, the random coefficients are described by the corresponding covariance function 
\begin{align}\label{Cov:Gauss}
	C_{z} (\boldsymbol{x},\boldsymbol{y}) = \kappa^2_{z} \prod_{k=1}^{d}  e^{-\left| x_k -y_k \right|/\ell_k } \quad \forall \boldsymbol{x},\boldsymbol{y} \in \mathcal{D} \subset \mathbb{R}^d,
\end{align}
with the  correlation length $\ell_k$ and standard deviation $\kappa_z$. In the KL representation \eqref{eqn:kltrun}, the exact eigenpair $(\lambda_j, \phi_j)$ of the covariance 
function \eqref{Cov:Gauss} is used; see, e.g., \cite{GJLord_CEPowell_TShardlow_2014} for the  computation of analytical values. We also note that the contribution of the 
truncation error is not taken into account for the sake of simplicity. In this study,  we consider the uniform distribution of random variables over the interval
$[-\sqrt{3},\sqrt{3}]$, that is,  $ \xi_i \sim \mathcal{U}[-\sqrt{3},\sqrt{3}]$ for $i=1,\ldots, N$.  Hence, Legendre orthogonal 
polynomial basis in $\Gamma$ and the piecewise linear finite elements in $\mathcal{\cD}$ are used to compute the state variable $y$ and the adjoint state $p$. For the 
computation of the control $u$, we also utilize linear finite elements on the boundary $\partial \mathcal{\cD}$. Further, we assume that there is no partition of $\Gamma$ and we only increase the polynomial degree $Q$.

\subsection{One dimensional unconstrained problem}
\label{sec:1d_example}

As a first benchmark example, we consider an unconstrained optimal control problem, that is, $U^{\textup{ad}} = L^2(\cD)$, in the one-dimensional spatial 
domain $\mathcal{D}=[0,1]$. Since each probability density function is uniform on $\Gamma_i$, the joint density function $\rho(\xi)$ is $ \frac{1}{(2\sqrt{3})^N}$. The optimal 
control problem we consider reads
\begin{eqnarray*}
	\min \limits_{u \in U^{\textup{ad}}} \; \mathcal{J}(y,u) := \frac{1}{2} \int_{[-\sqrt{3},\sqrt{3}]^N} \frac{1}{(2\sqrt{3})^N}  \int_0^1 (y-y^d)^2 \, \d x \, \, \d\xi 
        +  \frac{\alpha}{2} \big( u(0)^2 + u(1)^2 \big)
\end{eqnarray*}
subject to 
  \begin{align*}
    &-(a_N(x,\xi) \, y'(x, \xi))' = f_N(x,\xi)    &&             \text{for}\ (x,\xi) \in  [0,1]         \times [-\sqrt{3},\sqrt{3}]^N,\\
    & y(0, \xi) =  u(0) \qquad  y(1, \xi) =  u(1)   &&     \text{for}\ \xi  \in [-\sqrt{3},\sqrt{3}]^N,
  \end{align*}  
and we choose the source term $f_N(x,\xi) = 1 + \kappa \sum_{k=1}^{N}\sqrt{\lambda_k}\,\phi_k(x)\,\xi_k$, with $(\lambda_k,\phi_k)$ the eigenpairs of \eqref{Cov:Gauss}, regularization parameter $\alpha=10^{-2}$, desired state $y^d = \sin(\pi \,x) + \sin(2\pi \,x)$, and the diffusion coefficient $a_N(x,\xi) = 1 + \kappa\,\sum_{k=1}^{N} \sqrt{\lambda_k}\,\phi_k(x)\,\xi_k$. The correlation length is $\ell=1$ and the standard deviation $\kappa=0.5$.

\begin{figure}
	\subfloat[Convergence w.\,r.\,t.\ $h$ ($Q$ fixed)\label{fig:errors_1d_example_h}]{\includegraphics[width=.48\textwidth]{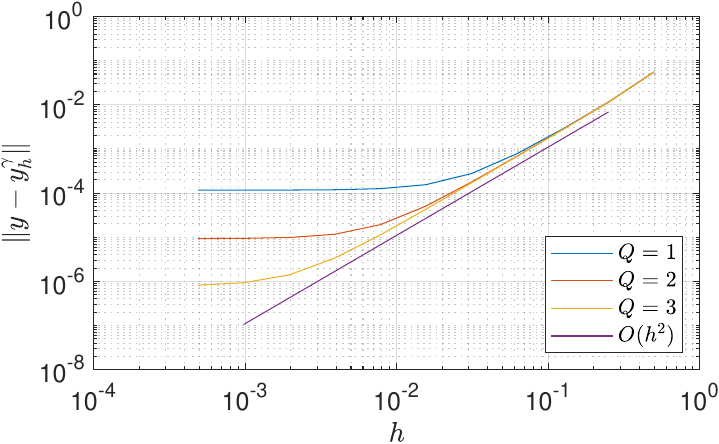}}
	\subfloat[Convergence w.\,r.\,t.\ $Q$ ($h$ fixed)\label{fig:errors_1d_example_Q}]{\includegraphics[width=.48\textwidth]{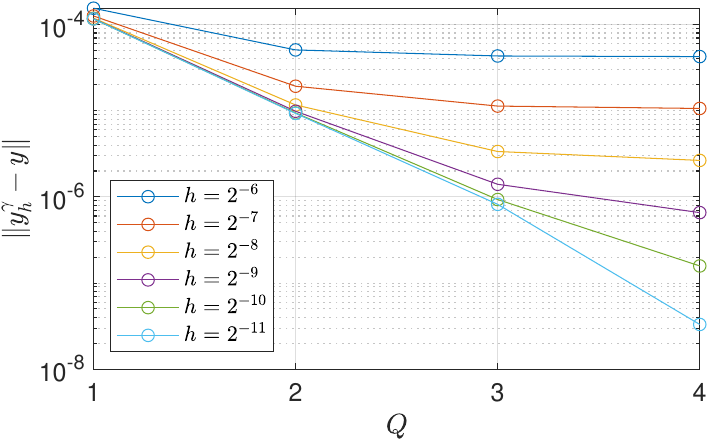}}
	\caption{Error plots measuring $\|y-y_h^\gamma\|_{L^2(\Gamma; L^2(\cD))}$ for the example from Section~\ref{sec:1d_example}.}
	\label{fig:errors_1d_example}
\end{figure}
\begin{table}
\centering
\begin{tabular}{lccccccccccc}
    \toprule
     order $Q$ & \multicolumn{11}{c@{}}{mesh size $h$} \\
     \cmidrule(l){2-12}
     & $2^{-2}$ & $2^{-3}$ & $2^{-4}$ & $2^{-5}$ & $2^{-6}$ & $2^{-7}$ & $2^{-8}$ & $2^{-9}$ & $2^{-10}$ & $2^{-11}$ & $2^{-12}$ \\
     \midrule
    1  &  20  &  29  &  34  &  38  &  40  &  44  &  44  &  48  &  44  &  44  &  48  \\
    2  &  29  &  35  &  42  &  44  &  51  &  54  &  54  &  52  &  55  &  56  &  52  \\
    3  &  33  &  41  &  47  &  49  &  57  &  56  &  62  &  58  &  62  &  62  &  57  \\
    4  &  35  &  44  &  51  &  56  &  63  &  63  &  67  &  67  &  66  &  67  &  66  \\
    \bottomrule
\end{tabular}
\caption{Iteration numbers for \textsc{minres} to achieve a relative tolerance of $10^{-10}$ in the example from Section~\ref{sec:1d_example}.}
\label{tab:ex_1d_iterations}
\end{table}
The exact solution to our problem is unknown and in order to compute the approximation error $\|y-y_h^\gamma\|_{L^2(\Gamma;L^2(\cD))}$ we use, instead of the analytical solution $y$, a reference solution on the finest grid and highest polynomial degree ($h=2^{-12}$, $Q=4$). The measured errors are illustrated in Figure~\ref{fig:errors_1d_example}. In Figure~\ref{fig:errors_1d_example_h} it can be observed that the state converges with order $2$ with respect to the spatial discretization until the stochastic error dominates the overall error. The observed convergence order is even higher than predicted in Theorem~\ref{thm:final_estimate} which is due to the fact that we have proved a worst-case estimate only which is sharp in the case that the solution is as regular as one would expect on general convex polygonal domains in space dimension $d\ge 2$. This fact does not play a role in a 1D example as there are no singularities involved stemming from the geometry of the computational domain. In Figure~\ref{fig:errors_1d_example_Q} the convergence with respect to the polynomial degree $Q$ of the discretization in the stochastic space is illustrated. We indeed observe the exponential convergence unless the spatial discretization error not dominates in the total error. Both results confirm the theoretically predicted convergence behavior from Theorem~\ref{thm:final_estimate}. 

The optimality system \eqref{eq:opt_cond_matrix_vector} has been  solved with a \textsc{minres} method in \textsc{matlab} and we have used the preconditioner derived in Section~\ref{sec:preconditioning}. In Table~\ref{tab:ex_1d_iterations} we report the iteration numbers of \textsc{minres} required to achieve the relative tolerance of $10^{-10}$. The preconditioner is not perfectly robust with respect to refinement in the spatial and stochastic space, but the iteration numbers remain in a moderate order of magnitude. It is not reported  but further numerical tests have shown that the iteration numbers will not change when choosing the regularization parameter $\alpha$ smaller.

\subsection{Two dimensional unconstrained problem}
\label{sec:2d_example}

Next, we consider a two dimensional unconstrained problem with the following data 
\begin{equation*}
	\cD:=[0,1]^2,\quad y^d(\bx) = x_1^2\,\sin(\pi\,x_2),\quad \bar f(\bx) = 0,\quad \bar a(\bx) = 1,\quad \alpha = 0.5,\quad N=10.
\end{equation*}
The boundary noise $b_N$ follows the expansion stemming from the Brownian bridge
\begin{equation*}
	b_N(\bx(s),\xi) = \sum_{k=1}^N \frac{\sqrt2}{k\,\pi}\,\sin(k\,\pi\,s)\,\xi_k,
\end{equation*}
with $[0,1]\ni s\mapsto \bx(s)\in \partial\cD$ a parametrization of the boundary of $\cD$. Here, the correlation length and the standard deviation are  $\ell=0.9$ and $\kappa=0.5$, respectively. The initial mesh for the computation consists of 4 triangles having their base at one of the edges of $\partial\cD$ and the opposite vertex in $(\frac12,\frac12)$. Then, we use twice bisection of each triangle across their longest edge to obtain the mesh on the next refinement level. Again, the exact solution is unknown and we compute the error using the solution for $Q=4$ and $h=2^{-8}$ as good approximation for $y$. The expected values of the optimal state $\E[y_h^\gamma]$ and adjoint state $\E[p_h^\gamma]$ are depicted in Figure~\ref{fig:solution_2d_example}. We observe the typical behaviour for the solution of Dirichlet control problems, that the solution becomes zero in the (convex) corners.

As can be seen in Figure~\ref{fig:errors_2d_example_h} the error in the control $\|u-u_h\|_{L^2(\partial\cD)}$ converges with an approximate rate of $3/2$. Note that even for the deterministic problem the best-possible convergence rate for the control is $1$ provided that the input data is smooth and the computational domain has no corners with opening angle larger than $120^\circ$; see \cite{TApel_MMateos_JPfefferer_ARosch_2018}. The improved rate of $3/2$ is due to the fact that the meshes in our computation possess a so-called superconvergence property allowing for higher convergence rates; see, e.g., \cite[Section 4]{TApel_MMateos_JPfefferer_ARosch_2018} for a detailed discussion on these effects.
For the convergence with respect to the polynomial degree $Q$ of the discretization in the stochastic space, an exponential convergence is observed unless the spatial discretization error not dominates in the total error; see, Figure~\ref{fig:errors_2d_example_Q}.

\begin{figure}[tb]
	\subfloat[Optimal state ${\E[y_h^\gamma]}$]{\includegraphics[width=.48\textwidth]{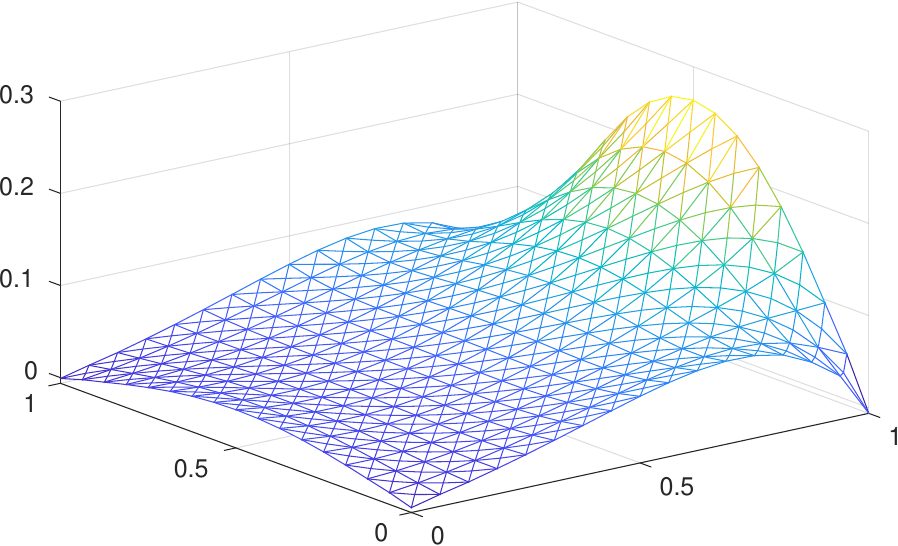}}%
	\subfloat[Optimal adjoint state ${\E[p_h^\gamma]}$]{\includegraphics[width=.48\textwidth]{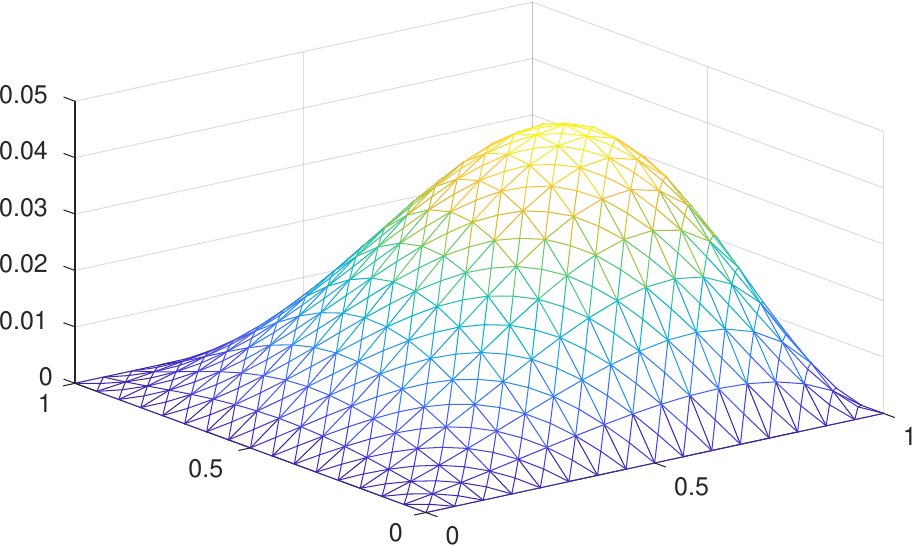}}
	\caption{Optimal state ${\E[y_h^\gamma]}$ and optimal adjoint state ${\E[p_h^\gamma]}$ with $Q=2$ and $h=2^{-5}$ for the optimal control problem from Section~\ref{sec:2d_example}.}
	\label{fig:solution_2d_example}
\end{figure}
\begin{figure}[tb]
	\subfloat[Convergence w.\,r.\,t.\ $h$ ($Q$ fixed)\label{fig:errors_2d_example_h}]{\includegraphics[width=.48\textwidth]{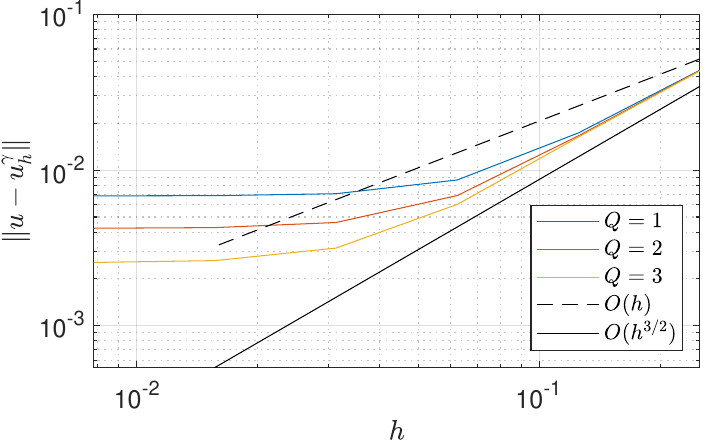}}
	\subfloat[Convergence w.\,r.\,t.\ $Q$ ($h$ fixed)\label{fig:errors_2d_example_Q}]{\includegraphics[width=.48\textwidth]{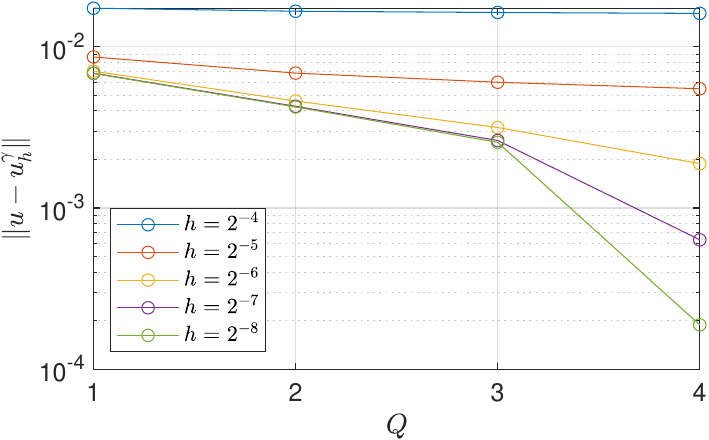}}
	\caption{Error plots measuring $\|u-u_h^\gamma\|_{L^2(\partial\cD)}$  for the example from Section~\ref{sec:2d_example}.}
	\label{fig:errors_2d_example}
\end{figure}

\subsection{Two dimensional constrained problem}
\label{sec:2d_example_const}

\begin{figure}[tb]
	\subfloat[Optimal state ${\E[y_h^\gamma]}$]{\includegraphics[width=.48\textwidth]{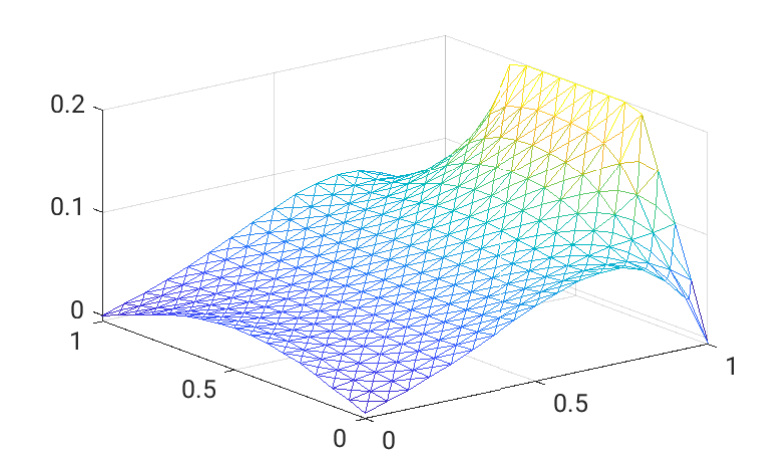}}%
	\subfloat[Optimal adjoint state ${\E[p_h^\gamma]}$]{\includegraphics[width=.48\textwidth]{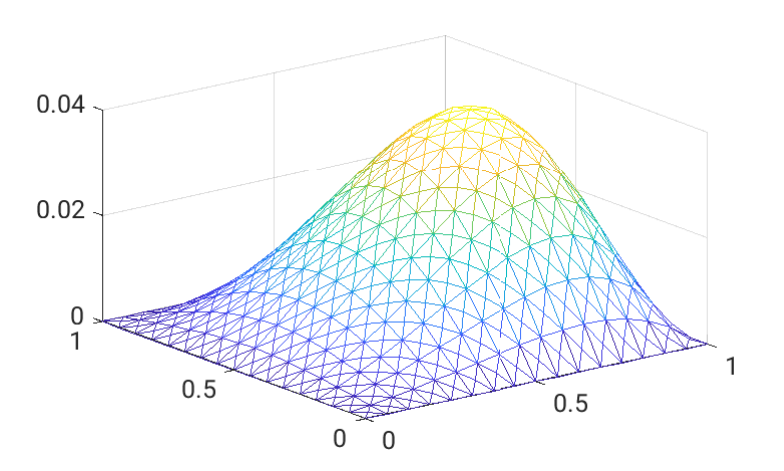}}
	\caption{Optimal state ${\E[y_h^\gamma]}$ and optimal adjoint state ${\E[p_h^\gamma]}$ with $Q=2$ and $h=2^{-5}$ for the optimal control problem from Section~\ref{sec:2d_example_const}.}
	\label{fig:solution_2d_constr_example}
\end{figure}

In a final experiment we consider the problem from Section~\ref{sec:2d_example} imposing additionally the upper control bound $u\le u_b = 0.2$. We use the primal-dual active set strategy, explained in Section~\ref{sec:constrained_problem}, and the \textsc{BiCGstab} method from \textsc{Matlab} with a Schur complement preconditioner to solve the linear systems \eqref{eq:dc6}. We compute the solution on a sequence of meshes with mesh size $h_\ell=2^{-\ell}$, $\ell=1,\ldots,8$, and polynomial degrees for the stochastic basis functions $Q=1,2,3$, whereas  the solution on $h=2^{-8}$ and $Q=3$ is used as reference solution for the error computation. The expected values of the optimal state and adjoint state are illustrated in Figure~\ref{fig:solution_2d_constr_example}. The resulting error propagation can be seen in Figure~\ref{fig:errors_2d_example_constrained}. As before, we see exponential convergence with respect to $Q$ and polynomial convergence with rate $3/2$ with respect to $h$.

\begin{figure}[tb]
	\subfloat[Convergence w.\,r.\,t.\ $h$ ($Q$ fixed)\label{fig:errors_2d_constr_example_h}]{\includegraphics[width=.48\textwidth]{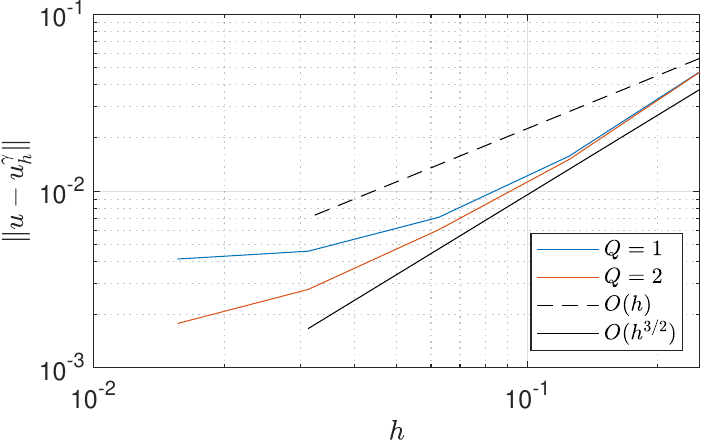}}
	\subfloat[Convergence w.\,r.\,t.\ $Q$ ($h$ fixed)\label{fig:errors_2d_constr_example_Q}]{\includegraphics[width=.48\textwidth]{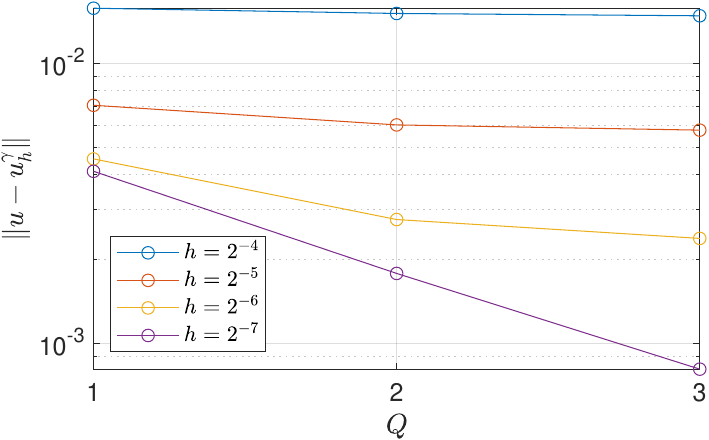}}
	\caption{Error plots measuring $\|u-u_h^\gamma\|_{L^2(\partial\cD))}$ for the example from Section~\ref{sec:2d_example_const}.}
	\label{fig:errors_2d_example_constrained}
\end{figure}



\section{Concluding remarks}\label{sec:conclusion} 

In this article, we have explored the application of the stochastic Galerkin method to Dirichlet boundary control problems with uncertain data, aiming to minimize the expected tracking cost functional under deterministic control constraints. By utilizing the finite dimensional noise assumption, we transform the stochastic state problem in its very weak form into a high-dimensional deterministic problem and then employ standard piecewise linear and continuous finite elements for all unknowns. A rigorous error analysis is provided and  numerical examples demonstrate the validation of the theoretical results and the effectiveness of the proposed preconditioners for both unconstrained and constrained scenarios. As a future study, the obtained results can be extended to a broader range of risk measures \cite{JORoysets_2022,RTRockafellar_JORoyset_2015} or integrated with various types of constraints, such as probabilistic or almost sure state constraints \cite{CGeiersbach_RHenrion_PPerezAros_2025,DPKouri_MStaudigl_TMSurowiec_2023}.

\section*{Acknowledgements}
HY gratefully acknowledges the research support provided by the Scientific and Technological Research Council of Turkey (TUBITAK) under the program TUBITAK 2219-International Postdoctoral Research Fellowship (Project No: 1059B192302207) and would like to thank the Max Planck Institute for Dynamics of Complex Technical Systems, Magdeburg for its excellent hospitality.

\bibliographystyle{plain}
\bibliography{references,more_references}

\begin{thebibliography}{10}

\bibitem{RAAdams_1975}
R.~A. Adams.
\newblock {\em Sobolev Spaces}.
\newblock Academic Press, Orlando, San Diego, New-York, 1975.

\bibitem{AAAli_EUllman_MHinze_2017}
A.~A. Ali, E.~Ullmann, and M.~Hinze.
\newblock Multilevel {Monte Carlo} analysis for optimal control of elliptic
  {PDEs} with random coefficients.
\newblock {\em SIAM/ASA J. Uncertain. Quantif.}, 5:466--492, 2017.

\bibitem{DAFrench_JTKing_1993}
D.~A.~French anf J.~T.~King.
\newblock Analysis of a robust finite element approximation for a parabolic
  equation with rough boundary data.
\newblock {\em Math. Comput.}, 60:79--104, 1993.

\bibitem{TApel_MMateos_JPfefferer_ARosch_2015}
T.~Apel, M.~Mateos, J.~Pfefferer, and A.~R\"osch.
\newblock On the regularity of the solutions of {D}irichlet optimal control
  problems in polygonal domains.
\newblock {\em SIAM J. Control Optim.}, 53(6):3620--3641, 2015.

\bibitem{TApel_MMateos_JPfefferer_ARosch_2018}
T.~Apel, M.~Mateos, J.~Pfefferer, and A.~R\"osch.
\newblock Error estimates for {D}irichlet control problems in polygonal
  domains: quasi--uniform meshes.
\newblock {\em Math. Control and Relat. F.}, 8(1):217--245, 2018.

\bibitem{IBabuska_RTempone_GEZouraris_2004a}
I.~Babu{\v{s}}ka, R.~Tempone, and G.~E. Zouraris.
\newblock {G}alerkin finite element approximations of stochastic elliptic
  partial differential equations.
\newblock {\em SIAM J. Numer. Anal.}, 42(2):800--825, 2004.

\bibitem{AVBarel_SVandewalle_2019}
A.~V. Barel and S.~Vandewalle.
\newblock Robust optimization of {PDEs} with random coefficients using a
  multilevel {Monte Carlo} method.
\newblock {\em SIAM/ASA J. Uncertain. Quantif.}, 7(1):174--202, 2019.

\bibitem{PBenner_AOnwunto_MStoll_2016}
P.~Benner, A.~Onwunta, and M.~Stoll.
\newblock Block-diagonal preconditioning for optimal control problems
  constrained by {PDEs} with uncertain inputs.
\newblock {\em SIAM. J. Matrix Anal. $\&$ Appl.}, 37:491--518, 2016.

\bibitem{MBerggren_2004a}
M.~Berggren.
\newblock Approximations of very weak solutions to boundary-value problems.
\newblock {\em SIAM J. Numer. Anal.}, 42(2):860--877, 2004.

\bibitem{MBergounioux_KIto_KKunisch_1999a}
M.~Bergounioux, K.~Ito, and K.~Kunisch.
\newblock Primal-dual strategy for constrained optimal control problems.
\newblock {\em SIAM J. Control Optim.}, 37(4):1176--1194, 1999.

\bibitem{ABorzi_2010a}
A.~Borz{\`{\i}}.
\newblock Multigrid and sparse-grid schemes for elliptic control problems with
  random coefficients.
\newblock {\em Comput. Vis. Sci.}, 13:153--160, 2010.

\bibitem{ABorzi_VSchulz_CSchillings_GvonWinckel_2010a}
A.~Borz{\`{\i}}, V.~Schulz, C.~Schillings, and G.~{von Winckel}.
\newblock On the treatment of distributed uncertainties in {PDE} constrained
  optimization.
\newblock {\em GAMM Mitt.}, 33(2):230--246, 2010.

\bibitem{ABorzi_GvonWinckel_2009a}
A.~Borz{\`{\i}} and G.~{von Winckel}.
\newblock Multigrid methods and sparse-grid collocation techniques for
  parabolic optimal control problems with random coefficients.
\newblock {\em SIAM J. Sci. Comput.}, 31(3):2172--2192, 2009.

\bibitem{RHCameron_WTMartin_1947a}
R.~H. Cameron and W.~T. Martin.
\newblock The orthogonal development of non-linear functionals in series of
  {F}ourier-{H}ermite functionals.
\newblock {\em Ann. of Math. (2)}, 48:385--392, 1947.

\bibitem{ECasas_1985}
E.~Casas.
\newblock {$L^2$} estimates for the finite element method for the {D}irichlet
  problem with singular data.
\newblock {\em Numer. Math.}, 47:627--632, 1985.

\bibitem{ECasas_JPRaymond_2006b}
E.~Casas and J.-P. Raymond.
\newblock Error estimates for the numerical approximation of {D}irichlet
  boundary control for semilinear elliptic equations.
\newblock {\em SIAM J. Control Optim.}, 45(5):1586--1611, 2006.

\bibitem{ECasas_JPRaymond_2006a}
E.~Casas and J.-P. Raymond.
\newblock The stability in {$W\sp {s,p}(\Gamma)$} spaces of {$L\sp
  2$}-projections on some convex sets.
\newblock {\em Numer. Funct. Anal. Optim.}, 27(2):117--137, 2006.

\bibitem{PCiloglu_HYucel_2023}
P.~\c{C}\.{i}lo\u{g}lu and H.~Y\"ucel.
\newblock Stochastic discontinuous {G}alerkin methods for robust deterministic
  control of convection diffusion equations with uncertain coefficients.
\newblock {\em Adv. Comput. Math.}, 49(16), 2023.

\bibitem{PChen_AQuarteroni_GRozza_2013}
P.~Chen, A.~Quarteroni, and Gianluigi Rozza.
\newblock Stochastic optimal {R}obin boundary control problems of
  advection-dominated elliptic equations.
\newblock {\em SIAM J. Numer. Anal.}, 51:2700--2722, 2013.

\bibitem{PGCiarlet_1978a}
P.~G. Ciarlet.
\newblock {\em The Finite Element Method for Elliptic Problems}.
\newblock North--Holland, Amsterdam, New York, 1978.

\bibitem{KDeckelnick_AGunther_MHinze_2009}
K.~Deckelnick, A.~G\"unther, and M.~Hinze.
\newblock Finite element approximation of {D}irichlet boundary control for
  elliptic {PDEs} on two--and three--dimensional curved domains.
\newblock {\em SIAM J. Control Optim.}, 48:2798--2819, 2009.

\bibitem{OGErnst_EUllmann_2010}
O.~G. Ernst and E.~Ullmann.
\newblock Stochastic {G}alerkin matrices.
\newblock {\em SIAM J. Matrix Anal. Appl.}, 31:1848--1872, 2010.

\bibitem{SGarreis_MUlbrich_2017}
S.~Garreis and M.~Ulbrich.
\newblock Constrained optimization with low-rank tensors and applications to
  parametric problems with {PDEs}.
\newblock {\em SIAM J. Sci. Comput.}, 39:A25--A54, 2017.

\bibitem{CGeiersbach_RHenrion_PPerezAros_2025}
C.~Geiersbach, R.~Henrion, and P.~Pérez-Aros.
\newblock Numerical solution of an optimal control problem with probabilistic
  and almost sure state constraints.
\newblock {\em J. Optim. Theory Appl.}, 204(7), 2025.

\bibitem{MDGunzburger_HCLee_JLee_2011a}
M.~D. Gunzburger, H.-C. Lee, and J.~Lee.
\newblock Error estimates of stochastic optimal {N}eumann boundary control
  problems.
\newblock {\em SIAM J. Numer. Anal.}, 49(4):1532--1552, 2011.

\bibitem{MDGunzburger_SManservisi_2000a}
M.~D. Gunzburger and S.~Manservisi.
\newblock The velocity tracking problem for {N}avier--{S}tokes flows with
  boundary control.
\newblock {\em SIAM J. Control Optim.}, 39:594--634, 2000.

\bibitem{PAGuth_VKaarnioja_FYKuo_CSchillings_IHSloan_2021a}
P.~A. Guth, V.~Kaarnioja, F.~Y. Kuo, C.~Schillings, and I.~H. Sloan.
\newblock A quasi-{Monte Carlo} method for optimal control under uncertainty.
\newblock {\em SIAM/ASA J. Uncertain. Quantif.}, 9(2):354--383, 2021.

\bibitem{LSHou_JLee_HManouzi_2011}
L.~S. Hou, J.~Lee, and H.~Manouzi.
\newblock Finite element approximations of stochastic optimal control problems
  constrained by stochastic elliptic {PDEs}.
\newblock {\em J. Math. Anal. Appl.}, 384:87--103, 2011.

\bibitem{KKarhunen_1947}
K.~Karhunen.
\newblock \"{U}ber lineare {M}ethoden in der {W}ahrscheinlichkeitsrechnung.
\newblock {\em Ann. Acad. Sci. Fennicae. Ser. A. I. Math.-Phys.}, 1947(37):79,
  1947.

\bibitem{DPKouri_MHeinkenschloss_DRidzal_BGWaanders_2013}
D.~P. Kouri, M.~Heinkenschloss, D.~Ridzal, and B.~G. van Bloemen~Waanders.
\newblock A trust-region algorithm with adaptive stochastic collocation for
  {PDE} optimization under uncertainty.
\newblock {\em SIAM J. Sci. Comput.}, 35:A1847--A1879, 2013.

\bibitem{DPKouri_MStaudigl_TMSurowiec_2023}
D.~P. Kouri, M.~Staudigl, and T.~M. Surowiec.
\newblock A relaxation-based probabilistic approach for {PDE}-constrained
  optimization under uncertainty with pointwise state constraints.
\newblock {\em Comput. Optim. Appl.}, 85:441--478, 2023.

\bibitem{AKunoth_CSchwab_2016}
A.~Kunoth and C.~Schwab.
\newblock Sparse adaptive tensor {G}alerkin approximations of stochastic
  {PDE}-constrained control problems.
\newblock {\em SIAM/ASA J. Uncertain. Quantif.}, 4:1034--1059, 2016.

\bibitem{MLazar_EZuazua_2014}
M.~Lazar and E.~Zuazua.
\newblock Averaged control and observation of parameter-depending wave
  equations.
\newblock {\em C. R. Math. Acad. Sci. Paris}, 352:497--502, 2014.

\bibitem{HCLee_JLee_2013}
H.-C. Lee and J.~Lee.
\newblock A stochastic {G}alerkin method for stochastic control problems.
\newblock {\em Commun. Comput. Phys.}, 14:77--106, 2013.

\bibitem{MLoeve_1946}
M.~Lo{\`e}ve.
\newblock Fonctions al\'eatoires de second ordre.
\newblock {\em Revue Sci.}, 84:195--206, 1946.

\bibitem{GJLord_CEPowell_TShardlow_2014}
G.~J. Lord, C.~E. Powell, and T.~Shardlow.
\newblock {\em An Introduction to Computational Stochastic {PDEs}}.
\newblock Cambridge University Press, New York, 2014.

\bibitem{JMFrutos_MKessler_AMunch_FPeriago_2016}
J.~Martínez-Frutos, M.~Kessler, A.~Münch, and F.~Periago.
\newblock Robust optimal {R}obin boundary control for the transient heat
  equation with random input data.
\newblock {\em Int. J. Numer. Meth. Engng.}, 108:116--135, 2016.

\bibitem{MMateos_2018}
M.~Mateos.
\newblock Optimization methods for {D}irichlet control problems.
\newblock {\em Optimization}, 67:585--617, 2018.

\bibitem{SMay_RRannacher_BVexler_2013}
S.~May, R.~Rannacher, and B.~Vexler.
\newblock Error analysis for a finite element approximation of elliptic
  {D}irichlet boundary control problems.
\newblock {\em SIAM J. Control Optim.}, 51:2585--2611, 2013.

\bibitem{FNegri_AManzoni_GRozza_2015}
F.~Negri, A.~Manzoni, and G.~Rozza.
\newblock Reduced basis approximation of parametrized optimal flow control
  problems for the {S}tokes equations.
\newblock {\em Comput. Math. Appl.}, 69(4):319--336, 2015.

\bibitem{BOksendal_2003}
B.~{\O}ksendal.
\newblock {\em Stochastic Differential Equations}.
\newblock Springer-Verlag, Berlin, 2003.

\bibitem{JWPearson_JPestana_2020}
J.~W. Pearson and J.~Pestana.
\newblock Preconditioners for {Krylov} subspace methods: {An} overview.
\newblock {\em GAMM Mitt.}, 43(4):e202000015, 2020.

\bibitem{JPfefferer_MWinkler_2019}
J.~Pfefferer and M.~Winkler.
\newblock Finite element error estimates for normal derivatives on boundary
  concentrated meshes.
\newblock {\em SIAM J. Numer. Anal.}, 57(5):2043--2073, 2019.

\bibitem{MPorcelli_VSimoncini_MTani_2015}
M.~Porcelli, V.~Simoncini, and M.~Tani.
\newblock Preconditioning of active-set {N}ewton methods for {PDE}-constrained
  optimal control problems.
\newblock {\em SIAM J. Sci. Comput.}, 37(5):S472--S502, 2015.

\bibitem{RRannacher_RScott_1982}
R.~Rannacher and R.~Scott.
\newblock Some optimal error estimates for piecewise linear finite element
  approximations.
\newblock {\em Math. Comp.}, 38(158):437--445, 1982.

\bibitem{RTRockafellar_JORoyset_2015}
R.~T. Rockafellar and J.~O. Royset.
\newblock Engineering decisions under risk averseness.
\newblock {\em ASCE-ASME J. Risk Uncertain. Eng. Syst., Part A: Civ. Eng.},
  1(2):04015003, 2015.

\bibitem{ERosseel_GNWells_2012}
E.~Rosseel and G.~N. Wells.
\newblock Optimal control with stochastic {PDE} constraints and uncertain
  controls.
\newblock {\em Comput. Methods Appl. Mech. Engrg.}, 213/216:152--167, 2012.

\bibitem{JORoysets_2022}
J.~O. Roysets.
\newblock Risk-adaptive approaches to learning and decision making: {A} survey,
  2024.
\newblock arXiv:2212.00856.

\bibitem{StollWinkler2021}
M.~Stoll and M.~Winkler.
\newblock Optimal {D}irichlet control of partial differential equations on
  networks.
\newblock {\em Electron. Trans. Numer. Anal.}, 54:392--419, 2021.

\bibitem{HTiesler_RMKirby_DXiu_TPreusser_2012a}
H.~Tiesler, R.~M. Kirby, D.~Xiu, and T.~Preusser.
\newblock Stochastic collocation for optimal control problems with stochastic
  {PDE} constraints.
\newblock {\em SIAM J. Control Optim.}, 50(5):2659--2682, 2012.

\bibitem{FTroeltzsch_2010a}
F.~Tr{\"o}ltzsch.
\newblock {\em Optimal Control of Partial Differential Equations: {T}heory,
  Methods and Applications}, volume 112 of {\em Graduate Studies in
  Mathematics}.
\newblock American Mathematical Society, Providence, RI, 2010.

\bibitem{EUllmann_2010}
E.~Ullmann.
\newblock A {K}ronecker product preconditioner for stochastic {G}alerkin finite
  element discretizations.
\newblock {\em SIAM J. Sci. Comput.}, 32(2):923--946, 2010.

\bibitem{NWiener_1938a}
N.~Wiener.
\newblock The homogeneous chaos.
\newblock {\em Amer. J. Math.}, 60:897--938, 1938.

\bibitem{MWinkler_2020}
M.~Winkler.
\newblock Error estimates for variational normal derivatives and {D}irichlet
  control problems with energy regularization.
\newblock {\em Numer. Math.}, 144:413--445, 2020.

\bibitem{DXiu_GEKarniadakis_2002}
D.~Xiu and G.~Em Karniadakis.
\newblock The {W}iener--{A}skey polynomial chaos for stochastic differential
  equations.
\newblock {\em SIAM J. Sci. Comput.}, 24(2):619–644, 2002.

\bibitem{WZhao_MGunzburger_2023}
W.~Zhao and M.~Gunzburger.
\newblock Stochastic collocation method for stochastic boundary control of the
  {N}avier-{S}tokes equations.
\newblock {\em Appl. Math. Optim.}, 287:6, 2023.

\bibitem{EZuazua_2014}
E.~Zuazua.
\newblock Averaged control.
\newblock {\em Automatica}, 50(12):3077--3087, 2014.

\end{thebibliography}
\end{document}